\documentclass[12pt]{amsart}
\usepackage{geometry}
\usepackage{amsmath}
\usepackage{amsthm}
\usepackage{bbm}
\usepackage{amsfonts,amssymb,bm}
\usepackage{fancyhdr}
\usepackage{mathrsfs}
\usepackage{appendix}
\usepackage{graphicx}
\usepackage[all]{xy}
\usepackage{color}
\usepackage{appendix}
\usepackage{enumitem}
\usepackage{comment}

\usepackage{float} 
\usepackage{amsthm}

\newtheorem{thm}{Theorem}[section]
\newtheorem{defi}[thm]{Definition}
\newtheorem{lem}[thm]{Lemma}
\newtheorem{cor}[thm]{Corollary}

\newtheorem{rmk}[thm]{Remark}

\usepackage{times}
\usepackage{hyperref}
\def\Aut{{\rm Aut}}
\def\Hom{{\rm Hom}}
\def\ord{{\rm ord}}
\def\min{{\rm min}}
\def\max{{\rm max}}
\def\inf{{\rm inf}}
\def\sup{{\rm sup}}
\def\lim{{\rm lim}}
\def\limsup{{\rm lim\,sup}}
\def\liminf{{\rm lim\,inf}}
\def\dif{{\rm d}}
\def\interior{{\rm int}}
\def\Supp{{\rm Supp}}
\def\gr{{\rm gr}}
\def\Gr{{\rm Gr}}

\def\pr{{\rm pr}}
\def\wt{{\rm wt}}

\def\Rees{{\rm Rees}}
\def\Proj{{\rm Proj}}
\def\Spec{{\rm Spec}}

\def\triv{{\rm triv}}

\def\vol{{\rm vol}}
\def\nvol{{\rm \widehat{vol}}}
\def\Val{{\rm Val}}

\def\lct{{\rm lct}}

\def\LE{{\rm LE}}

\def\Ex{{\rm Ex}}
\def\mult{{\rm mult}}
\def\Fut{{\rm Fut}}

\def\Cone{{\rm Cone}}

\def\red{{\rm red}}
\def\div{{\rm div}}
\def\QM{{\rm QM}}
\def\LC{{\rm LC}}

\def\Reeb{{\mathbf{t}^+_\mathbb{R}}}
\def\dist{{\rm dist}}

\def\BP{\mathbf{P}}

\def\BO{\mathbf{O}}

\def\BL{\mathbf{L}}
\def\BD{\mathbf{D}}

\def\BJ{\mathbf{J}}

\def\BF{\mathbf{F}}

\def\B0{\mathbf{0}}
\def\Bt{\mathbf{t}}
\def\By{\mathbf{y}}
\def\Balpha{\boldsymbol{\alpha}}
\def\Bbeta{\boldsymbol{\beta}}

\def\bg{\bar{g}}
\def\bpsi{\bar{\psi}}
\def\bphi{\bar{\phi}}

\def\tX{\tilde{X}}

\def\tS{\widetilde{S}}
\def\tCS{\widetilde{\CS}}
\def\talpha{\tilde{\alpha}}
\def\tell{\tilde{\ell}}

\newcommand{\IA}{{\mathbb A}}

\newcommand{\IG}{{\mathbb G}}

\newcommand{\Ik}{{\mathbbm k}}

\newcommand{\IN}{{\mathbb N}}

\newcommand{\IQ}{{\mathbb Q}} 
\newcommand{\IR}{{\mathbb R}}

\newcommand{\IT}{{\mathbb T}}

\newcommand{\IZ}{{\mathbb Z}}

\newcommand{\CF}{{\mathcal F}}
\newcommand{\CG}{{\mathcal G}}

\newcommand{\CO}{{\mathcal O}}

\newcommand{\CR}{{\mathcal R}}
\newcommand{\CS}{{\mathcal S}}

\newcommand{\CV}{{\mathcal V}}

\newcommand{\CX}{{\mathcal X}}

\newcommand{\fa}{\mathfrak{a}}

\newcommand{\fm}{\mathfrak{m}}
\newcommand{\fv}{\mathfrak{v}}

\newcommand{\seq}{\subseteq}
\newcommand{\la}{\langle}
\newcommand{\ra}{\rangle}

\newcommand{\bu}{\bullet}
\newcommand{\lam}{\lambda}
\newcommand{\D}{\Delta}

\newcommand{\vep}{\varepsilon}

\def\opi{{\overline{\pi}}}

\def\oY{{\overline{Y}}}
\def\oE{{\overline{E}}}

\title{K-polystability and reduced uniform K-stability of log Fano cone singularities}

\author{Linsheng Wang}
\address{Shanghai Center for Mathematical Sciences, Fudan University, Shanghai, 200438, China}
\curraddr{}
\email{linsheng\_wang@fudan.edu.cn}
\thanks{}
\date{}
\dedicatory{}

\begin{document}

\maketitle

\begin{abstract}
We prove that a log Fano cone $(X,\Delta,\xi_0)$ satisfying $\delta_\mathbb{T}(X,\Delta,\xi_0)\ge 1$ is K-polystable for normal test configurations if and only if it is K-polystable for special test configurations. We also establish the reduced uniform K-stability of $(X,\Delta,\xi_0)$ and show that it is equivalent to K-polystability. 

\end{abstract}


\section{Introduction}
\label{Section: Introduction}

The algebraic study of local K-stability theory has made great progress in recent years. It centers around the so-called {\it stable degeneration conjecture} (see \cite{Li18,LX18}): Let $x\in(X=\Spec(R),\D)$ be a Kawamata log terminal (klt) singularity. Then there exists a unique (up to rescaling) quasi-monomial valuation $v_0$ minimizing the normalized volume function $\nvol:\Val_{X,x}\to \IR_{\ge 0} \cup \{+\infty\}$ such that $\Gr_{v_0}R$ is finitely generated and the degeneration $(X_0,\D_0,\xi_0)$ induced by $v_0$ is a K-semistable log Fano cone. This conjecture is widely studied and solved now. The K-semistability of $(X_0,\D_0,\xi_0)$ was established by \cite{LX18} and \cite{XZ20}.  The existence and quasi-monomialness of $v_0$ were proved by \cite{Blu18} and \cite{Xu19} respectively. The uniqueness of $v_0$ was given by \cite{XZ20,BLQ22}. Finally, it was proved by \cite{XZ22} that $\Gr_{v_0}R$ is finitely generated, which finished the proof of the conjecture. 


By \cite{LWX18}, the K-semistable log Fano cone $(X_0,\D_0,\xi_0)$ has a unique (up to isomorphism) special degeneration $(X_p,\D_p,\xi_0)$, which is K-polystable for special test configurations (see \cite[Definition 2.23]{LWX18}). By the local YTD-type theorem (see \cite{CS18} for $\D=0$ and \cite{Li21} for general case), any log Fano cone $(Y,\D_Y,\xi_Y)$ admits a weak Ricci-flat K\"ahler cone metric if and only if it is K-polystable for $\IQ$-Gorenstein test configurations. One may expect that $(X_p,\D_p,\xi_0)$ is indeed K-polystable for $\IQ$-Gorenstein test configurations, or more generally, for normal test configurations. 
 
In the global setting, this was conjectured by \cite{Tia97} that a Fano variety is K-semi/polystable for normal test configurations if and only if it is K-semi/polystable for special test configurations. This was proved in \cite{LX14} by running MMP on families of Fano varieties over curves to decent the generalized Futaki invariants. 

In this paper, we solve the above problem for K-polystability of log Fano cones using the higher rank finite generation theory developed by \cite{LXZ22,XZ22}. 
The problem for K-semistability of log Fano cones was recently solved by Liu and Wu \cite{LW24}. Based on the non-Archimedean pluripotential theory of Fano cone singularities developed by \cite{Wu22,Wu24}, they introduced the functional $\BL^{\rm NA}$ and gave a new formulation of the Ding invariant $\BD^{\rm NA}$ (\cite[Definition A.1]{LWX18}). They showed that $\Fut\ge \BD^{\rm NA}$ for any normal test configuration, and the equality holds for weakly special test configurations, hence for special test configurations. In particular, we have
\begin{eqnarray}
\text{Ding-ps for normal TC} &\Rightarrow&
\text{K-ps for normal TC} \\ &\Rightarrow&
\text{K-ps for special TC} \,\,\,\Leftrightarrow\,\,\, \nonumber
\text{Ding-ps for special TC}. 
\end{eqnarray}
To show the equivalence of these polystability conditions, it suffices to prove the following. 
\begin{thm}
\label{Theorem: Intro. Ding-ps for special test configurations implies Ding-ps}
Let $(X,\D,\xi_0)$ be a log Fano cone satisfying $\delta_\IT(X,\D,\xi_0)\ge 1$. If it is Ding-polystable for special test configurations, then it is Ding-polystable for normal test configurations. 
\end{thm}

We will use another form of the Ding invariant (with respect to the language of {\it filtrations}, see \cite{BLQ22}), which suits our computation better (analogous to \cite[Chapter 3]{Xu24}). For simplicity of notation, we will use $\BD$ instead of $\BD^{\rm NA}$. Let $(X,\D,\xi_0)$ be a log Fano cone with a $\IT$-action and $\CF$ be a $\IT$-invariant linearly bounded filtration on it. We define the {\it Ding invariant} of $\CF$ by 
\begin{eqnarray}
\BD_{X,\D,\xi_0}(\CF) := \lct(X,\D;\CF) - A(\xi_0)S(\xi_0;\CF). 
\end{eqnarray}
We will show that, in the quasi-regular case, this is equal to the Ding invariant of $\CF$ on the anti-canonical ring of the quotient log Fano pair. Hence it is the same with the Ding invariants defined by \cite{LWX18} and \cite{LW24}. With the above form of Ding invariant, we get another proof of the following characterization of Ding-semistability. 
\begin{thm}
\label{Theorem. Intro. Ding-ss for filtrations equiv. delta>=1}
Let $(X,\D,\xi_0)$ be a log Fano cone. It is $\IT$-equivariantly Ding-semistable for filtrations if and only if $\delta_\IT(X,\D,\xi_0)\ge 1$. 
\end{thm}
We will say that $(X,\D,\xi_0)$ is Ding-semistable if it is $\IT$-equivariantly Ding-semistable for filtrations. 
One of the key steps in the proof of Theorem \ref{Theorem: Intro. Ding-ps for special test configurations implies Ding-ps} is the following local version of the optimal destabilization theorem. 
\begin{thm}
\label{Theorem. Local optimal destablization}
Let $(X,\D,\xi_0)$ be a log Fano cone. If $\delta_{\IT}(X,\D,\xi_0)$ admits a quasi-monomial minimizer $v$, then $v$ is a Koll\'ar valuation, and there exists a Koll\'ar component $E$ minimizing $\delta_{\IT}(X,\D,\xi_0)$. 
\end{thm}

Moreover, we establish the reduced uniform K/Ding-stability of log Fano cones and show that they are equivalent to K/Ding-polystability. 

\begin{thm}
\label{Theorem: Intro. Ding-ps for special test configurations implies reduced uniform Ding}
Let $(X,\D,\xi_0)$ be a log Fano cone singularity satisfying $\delta_\IT(X,\D,\xi_0)\ge 1$. If it is Ding-polystable for special test configurations, then it is reduced uniformly Ding-stable for filtrations. 
\end{thm}


\begin{cor}
\label{Corollary. Intro. all equivalent}
Let $(X,\D,\xi_0)$ be a log Fano cone singularity satisfying $\delta_\IT(X,\D,\xi_0)\ge 1$. Then the following are all equivalent: 
\begin{enumerate}
\item it is K-polystable for special test configurations; 
\item it is K-polystable for normal test configurations; 
\item it is Ding-polystable for special test configurations; 
\item it is Ding-polystable for normal test configurations; 
\item it is reduced uniformly K-stable for special test configurations; 
\item it is reduced uniformly K-stable for normal test configurations; 
\item it is reduced uniformly Ding-stable for special test configurations; 
\item it is reduced uniformly Ding-stable for normal test configurations; 
\item it is reduced uniformly Ding-stable for filtrations. 
\end{enumerate} 
\end{cor}

Since $\Fut\ge \BD$ for any normal test configuration and the equality holds for special test configurations, we have naturally
$(4)\Rightarrow (2) \Rightarrow (1) \Leftrightarrow (3)$ and $(8)\Rightarrow (6) \Rightarrow (5) \Leftrightarrow (7)$. On the other hand, one may easily show that reduced uniform Ding-stability implies Ding-polystability, that is, $(8)\Rightarrow (4)$ hold. 
Since the space of filtrations is larger than the space of test configurations, we have $(9)\Rightarrow (8)$. 
We need to prove the non-trivial implication $(3)\Rightarrow (9)$, 
that is, Theorem \ref{Theorem: Intro. Ding-ps for special test configurations implies reduced uniform Ding}. 

We sketch the proof of Theorem \ref{Theorem: Intro. Ding-ps for special test configurations implies reduced uniform Ding}. For any linearly bounded filtration $\CF$, there exists a sequence of finitely generated filtrations $\CF_m$ approximating it (Definition \ref{Definition. approximating sequences}). We first establish the convergence of $S$ and $\BJ_\IT$ (Theorem \ref{Theorem. S convergence. approximating sequences} and \ref{Theorem. convergence of J_T, approximating sequence}). Then the question is reduced to checking reduced uniform Ding stability for finitely generated filtrations. The problem is further reduced to checking for Koll\'ar components (Corollary \ref{Corollary. reduced uniform Ding for filtrations and finitely generated filtrations}). Assuming the contrary, there exists a sequence of $\IT$-invariant Koll\'ar components $v_i$ such that $\BD(v_i)/\BJ_\IT(v_i) \to 0$. Using the argument in \cite{Xu19,XZ19}, one may show that $v_i$ converge to a $\IT$-invariant quasi-monomial valuation $v$ such that $\BD(v)=0$ and $v$ is not of product type. Hence $v$ is a Koll\'ar valuation and is not of product type. Perturbing $v$ to a Koll\'ar component $w$, we still have $\BD(w)=0$ and $w$ is not of product type. Then $w$ induces a non-product type special test configuration with vanishing $\BD$, which is a contradiction.




The paper is organized as follows. In Section \ref{Section: Preliminaries} we collect some results in the local filtration theory and recall the basic properties of log Fano cone singularities. In Section \ref{Section. Polystability}, we first define local K/Ding-semi/polystability. Then recall the higher rank finite generation theory developed by \cite{LXZ22,XZ22} and prove the local optimal destablization theorem (Theorem \ref{Theorem. Local optimal destablization}). At the end of the section, we prove Theorem \ref{Theorem: Intro. Ding-ps for special test configurations implies Ding-ps}. In Section \ref{Section. Relations between K/Ding-stability}, we establish the inequality $\Fut\ge \BD$ for normal test configurations by perturbing the Reeb vector $\xi_0$. Finally, we introduce the notion of reduced uniform K/Ding-stability of a log Fano cone in Section \ref{Section. Reduced uniform stability}, and show that they are equivalent to K/Ding-polystability. In the appendix, we establish a Blum-Jonsson type uniform estimate of local S-invariants using local Okounkov bodies.

\noindent {\bf Acknowledgments}. The author would like to thank Jiyuan Han, Yuchen Liu, Junyao Peng, Lu Qi and Chuyu Zhou for helpful discussions. The author was partially supported by the NKRD Program of China (\#2025YFA1018100), (\#2023YFA1010600), (\#2020YFA0713200) and LNMS.

\section{Preliminaries}
\label{Section: Preliminaries}

We work over an algebraically closed field $\Ik$ of characteristic zero. We follow the standard terminology from \cite{KM98,Kol13}. 

Let $x\in(X=\Spec(R),\D)$ be a klt singularity of dimension $n$. We denote by $\Val_X$ the set of valuations over $X$, by $\Val_{X, x}\seq \Val_{X}$ the subset of valuations centered at $x$, by $\Val^*_{X}\seq \Val_X$ the subset of valuations with finite log discrepancies. If $x\in(X,\D)$ admits a torus $\IT$-action, then we denote by $\Val^{\IT}_{X}\seq \Val_X$ the subset of $\IT$-invariant valuations. We also set $\Val^*_{X,x}=\Val_X^*\cap\Val_{X,x}$ and $\Val^{\IT,*}_{X,x}=\Val_X^\IT\cap \Val^*_{X,x}$. 

A model $(Y,E)$ over a pair $(X,\D)$ consists of a projective birational morphism $\pi:Y\to X$ and a reduced divisor $E$ on $Y$. It is called a {\it log smooth} (resp. {\it toroidal}) model if $(Y,\Supp(E+\Ex(\pi)+\pi_*^{-1}\D))$ is simple normal crossing (resp. toroidal). 
Let $(X,\D)$ be a klt pair that is projective over a quasi-projective variety $U$ such that $-(K_X+\D)$ is ample. A {\it $\IQ$-complement} $\Gamma$ of $(X,\D)$ is an effective $\IQ$-divisor on $X$ such that $(X,\D+\Gamma)$ is log canonical and $K_X+\D+\Gamma \sim_\IQ 0$. In the case that $x\in(X,\D)$ is a klt singularity, we assume moreover that $x$ is a log canonical center of $(X,\D+\Gamma)$. We denote the space of log canonical places by 
$\LC(X,\D+\Gamma) := \{v\in \Val_X\mid A_{X,\D+\Gamma}(v)=0\}$.

\subsection{Filtrations}
Let $x\in(X=\Spec(R),\D)$ be a klt singularity of dimension $n$. We recall the definition of $\fm_x$-filtrations on $R$ introduced by \cite{BLQ22}. 

\begin{defi}\rm
\label{Definition. filtrations of klt singularities}
An \textit{($\fm_x$-)filtration} is a collection $\CF=\CF^\bu = \{\CF^\lam = \CF^\lam R\}_{\lambda \in \mathbb{R}_{>0}}$ of $\mathfrak{m}_x$-primary ideals of $R$ such that:
\begin{enumerate}
    \item $\CF^{\lambda} \supseteq \CF^{\mu}$ when $\lambda \le \mu$,
    \item $\CF^{\lambda} = \CF^{\lambda - \epsilon}$ when $0 < \epsilon \ll 1$, and
    \item $\CF^{\lambda} \cdot \CF^{\mu} \subseteq \CF^{\lambda + \mu}$ for any $\lambda, \mu \in \mathbb{R}_{>0}$.
\end{enumerate}
By convention, we set $\CF^0 := R$. For any $f\in R$, we denote by $\ord_\CF(f)=\max\{\lam\ge 0\mid f\in \CF^\lam\}. $
\end{defi}

For any filtration $\CF$ and $a>0$, we define the {\it $a$-rescaling} $a\CF$ of $\CF$ by 
\begin{eqnarray*}
(a\CF)^\lam = \CF^{\lam/a}, \quad \lam\in\IR_{\ge 0}.  
\end{eqnarray*} 
Let $v\in\Val_{X,x}$, we denote by $\CF_v$ the filtration 
\begin{eqnarray*}
\CF_v^\lam= \{f\in R \mid v(f)\ge \lam\}, \quad \lam\in\IR_{\ge 0}. 
\end{eqnarray*} 

\begin{rmk}\rm
\label{Remark. N-filtrations}
The definition is a local analog of a filtration of the section ring of a polarized variety in \cite{BHJ17}. 
For any graded ideal sequence $\fa_\bu=\{\fa_m\}_{m\in\IN}$ of $\fm_x$-primary ideals (satisfying $\fa_0=R, \fa_m\cdot \fa_{m'} \seq \fa_{m+m'}$ for any $m,m'\in\IN$), it is naturally a filtration by setting $\CF^\lam:=\fa_{\lceil \lam \rceil}$ for any $\lam\in \IR_{\ge0}$, which is also called a {\it $\IN$-filtration}. 
Conversely, any filtration $\CF^\bu=\{\CF^\lam\}_{\lam\in\IR_{\ge0}}$ gives rise to a $\IN$-filtration $\CF_\IN^{\lam} := \CF^{\lceil \lam \rceil}$ for any $\lam\in\IR_{\ge0}$. 
\end{rmk}

Let $\CF$ and $\CG$ be filtrations. 
Then $\CF$ is said to be {\it linearly bounded} by $\CG$ if there exists $C>0$ such that $\CF^{C\lam}\seq \CG^{\lam}$ for any $\lam>0$. If $\CG$ is also linearly bounded by $\CF$, we simply say that $\CF$ and $\CG$ are linearly bounded by each other. In this case, we set 
\begin{eqnarray}
\lam_\max(\CG;\CF) 
&:=& \sup\{\lam \in \IR_{>0} \mid \CF^{\lam m} \nsubseteq \CG^m, \exists m\in \IN\}, \\ 
\label{Eqnarray. Lam_max lam_min = 1}
\lam_\min(\CG;\CF) 
&:=& \inf\{\mu\in\IR_{>0} \mid \CG^m \nsubseteq \CF^{\mu m}, \exists m\in\IN \}
\,\,\,=\,\,\,\lam_\max(\CF;\CG)^{-1}. 
\end{eqnarray} 
We also have 
\begin{eqnarray*}
\lam_\max(\CG;\CF) 
&=& \inf\{\lam \in \IR_{>0} \mid \CF^{\lam m}\seq \CG^m, \forall m\in\IN \}, \\
\lam_\min(\CG;\CF) 
&=& \sup\{\mu\in\IR_{>0} \mid \CG^m \seq \CF^{\mu m}, \forall m\in\IN \}.
\end{eqnarray*} 
Note that the filtration $\{\fm^{\lceil \lam \rceil}\}_{\lam\in\IR_{>0}}$ is linearly bounded by any filtration $\CF$. Indeed, since $\CF^1$ is $\fm$-primary, there exists $N\in\IZ_{\ge 1}$ such that $\fm^N \seq \CF^1$. Hence for any $\lam\ge1$, $\fm^{N\lceil \lam \rceil} \seq (\CF^1)^{\lceil \lam \rceil} \seq \CF^{\lceil \lam \rceil} \seq \CF^\lam$. 
The filtration $\CF$ is called {\it linearly bounded} if it is linearly bounded by the filtration $\{\fm^{\lceil \lam \rceil}\}_{\lam\in\IR_{>0}}$. 
By the Izumi-type inequality \cite{Li18}, we see that $\CF_v$ is linearly bounded for any $v\in\Val_{X,x}^*$. 

Let $\CF$ be a filtration and $v\in \Val_X$. By \cite[Proposition 2.3]{JM12}, we define 
$$v(\CF):=  \mathop{\inf}_{m\in\IN} \frac{v(\CF^m)}{m}= \mathop{\lim}_{m\to \infty} \frac{v(\CF^m)}{m}. $$
We may define for each $m\in\IN$, 
$$\lam^{(m)}_\min(\CF;\CF_v) := \inf\{\mu \in \IR_{>0} \mid \CF^m \nsubseteq \CF_v^\mu\} = v(\CF^m), $$
and similarly show that 
$\lam_\min(\CF;\CF_v) = \inf_m m^{-1}\lam^{(m)}_\min(\CF;\CF_v)$. Hence 
\begin{eqnarray}
\label{Eqnarray. v(F) = lam_min(F; v)}
v(\CF) = \lam_\min(\CF;\CF_v). 
\end{eqnarray}

Assume that $\CF$ and $\CF_v$ are linearly bounded. We also define for each $m\in\IN$, 
$$\lam^{(m)}_\max(\CF_v;\CF) := \max\{\lam \in \IR_{>0} \mid \CF^{\lam } \nsubseteq \CF_v^m\} \,\,\,(\le m\lam_\max(\CF_v;\CF)). $$
The maximum exists since $\CF$ is left continuous. 
We have 
$$p\lam^{(m)}_\max(\CF_v;\CF) \le \lam^{(pm)}_\max(\CF_v;\CF).$$
Otherwise, we have $\CF^{p\lam^{(m)}_\max}\seq \CF_v^{pm}$ and $\CF^{\lam^{(m)}_\max}\nsubseteq \CF_v^m$ for some $m\in\IN$. Then there exists $ s\in \CF^{\lam^{(m)}_\max}\setminus \CF_v^m$. We get a contradiction since $v(s^p)=pv(s)<pm$ and $s^p\in\CF^{p\lam^{(m)}_\max}\setminus \CF_v^{pm}$. Hence
\begin{eqnarray}
\label{Eqnarray. lam_max^(m) convergence}
\lam_\max(\CF_v;\CF) =
\mathop{\sup}_{m\in \IN} \frac{\lam^{(m)}_\max(\CF_v;\CF)}{m} =
\mathop{\lim}_{m\to \infty} \frac{\lam^{(m)}_\max(\CF_v;\CF)}{m}. 
\end{eqnarray}

By \cite[Lemma 1.49]{Xu24}, we define 
$$
\lct(\CF) = \lct(X,\D;\CF) := \mathop{\sup}_{\lam>0} \lam\cdot\lct(X,\D;\CF^\lam) = \mathop{\lim}_{m\to \infty} m\cdot \lct(X,\D;\CF^m). 
$$
The {\it saturation} $\tilde{\CF}$ of a filtration $\CF$ is defined by 
$$\tilde{\CF}^\lam = \{f\in R \mid v(f)\ge \lam \cdot v(\CF), \forall v\in \Val_{X,x}^*\}. $$
It's clear that $\CF^\lam\seq\tilde{\CF}^\lam$. A filtration $\CF$ is called {\it saturated} if $\CF=\tilde{\CF}$.




In the following, we fix a valuation $v_0\in\Val_{X,x}^*$ and the corresponding linearly bounded $\IN$-filtration $\CF_0 = \CF_{v_0}$ on $R$. In applications, we will choose $v_0$ to be the normalized volume minimizer or $v_0=\wt_{\xi}$ for some $\xi$ in the Reeb cone. 
For any linearly bounded filtration $\CF$ ($\CF^{Cm}\seq \fm^{m}$), we define the {\it multiplicity} of $\CF$ by 
\begin{eqnarray}
\mult(\CF)
&:=& \mathop{\limsup}_{m\to \infty} \frac{\ell (R/\CF^{m})}{m^n/n!} \\
&\ge& \mathop{\lim}_{m\to \infty} \frac{\ell (R/\fm^{m/C})}{m^n/n!} 
= \mult(\fm^{\bu/C}) = C^n \mult(\fm) > 0, \nonumber
\end{eqnarray}
and the {\it relative multiplicity} of $\CF$ with respect to $\CF_0$ (\cite[Lemma-Definition 3.1]{XZ20}) by 
\begin{eqnarray}
\mult(\CF_0;\CF) 
:= \mathop{\lim}_{m\to \infty} \frac{\ell (\CF^m(R/\CF_0^m))}{m^n/n!} = \mult(\CF_0^\bu\cap\CF^\bu) -\mult(\CF^\bu), 
\end{eqnarray}
where the second equality follows from 
\begin{eqnarray*}
\ell(\CF^m(R/\CF_0^m)) = \ell(\CF^m/(\CF_0^m\cap \CF^m)) = \ell(R/(\CF_0^m \cap \CF^m)) -\ell(R/\CF^m). 
\end{eqnarray*}



Next we define the relative $S$-invariant $S(\CF_0;\CF)$ slightly generalizing that of \cite[Section 3.1]{XZ20}. Let $\CF$ be a linearly bounded filtration. By \cite[Lemma 2.1]{XZ20}, for any \( m \in \mathbb{N} \) there exist \( f_1, \cdots, f_{N_m} \in R \) (where \( N_m = \ell(R/\CF_0^{m}) \)) which are compatible with both \( \CF_0 \) and \( \CF \) such that their images \( \bar{f_i} \) form a basis of \( R/\CF_0^{m} \). We define 
\begin{eqnarray*}
\tS_m(\CF_0;\CF) := \sum_{i=1}^{N_m} \ord_\CF(f_i) = \sum_{\lam\ge 0} \lam \cdot \ell(\Gr_\CF^\lam(R/\CF_0^{m})), \quad 
S_m(\CF_0;\CF)
:=\frac{\tS_m(\CF_0; \CF)}{\tS_m(\CF_0; \CF_0)}. 
\end{eqnarray*}
which does not depend on the choice of $\{f_i\}$. 

If $\CF$ is a $\IN$-filtration, then 
\begin{eqnarray*}
\tS_m(\CF_0; \CF)
= \sum_{j=1}^{\lam_\max^{(m)}} j\cdot \ell(\Gr_\CF^j(R/\CF_0^{m}))
=\sum_{j=1}^{\lam_\max^{(m)}}\ell(\CF^j(R/\CF_0^{m})),
\end{eqnarray*}
where $\lam_\max^{(m)}=\lam_\max^{(m)}(\CF_0;\CF)$. 
For example, $\tS_m(\CF_0; \CF_0) = \sum_{i=0}^{m-1} i\cdot\ell(\Gr_{\CF_0}^i R)$. By \cite[Lemma 3.3]{XZ20}, the limit of $S_m(\CF_0;\CF)$ exists and we define
\begin{eqnarray}
\label{Eqnarray. Convergence of S_m}
\mathop{\lim}_{m\to\infty} S_m(\CF_0;\CF) 
=: S(\CF_0;\CF) 
= \frac{n+1}{n}\int_0^\infty \frac{\mult(\CF_0;t^{-1}\CF)}{\mult(\CF_0)} \dif t. 
\end{eqnarray}
Note that $\mult(\CF_0;t^{-1}\CF) = 0$ for any $t> \lam_\max(\CF_0;\CF)$. Hence ``$\infty$'' could be replaced by ``$\lam_\max(\CF_0;\CF)$'' in the above integration. 
We have another formulation of $S(\CF_0;\CF)$. Let 
\begin{eqnarray*}
S'_{m}(\CF_0;\CF)
=\frac{\tS_{m+1}(\CF_0; \CF)-\tS_{m}(\CF_0; \CF)}{\tS_{m+1}(\CF_0; \CF_0)-\tS_{m}(\CF_0; \CF_0)} 
= \frac{\sum_j j\cdot \ell(\Gr_{\CF}^j\Gr_{\CF_0}^m R)}{m\cdot\ell(\Gr_{\CF_0}^m R)}. 
\end{eqnarray*}
Then by Stolz's theorem, we also have $\lim_{m\to \infty}S'_m(\CF_0;\CF)=S(\CF_0;\CF)$.

In general, for any $\IR_{\ge0}$-filtration $\CF$, let $\CF_\IN$ be the corresponding $\IN$-filtration. Then for any $f\in R$, 
$$\ord_{\CF_\IN}(f) = \lfloor\ord_{\CF}(f)\rfloor \le \ord_{\CF}(f) < \ord_{\CF_\IN}(f)+1. $$Hence $\tS_m(\CF_0;\CF_\IN) \le \tS_m(\CF_0;\CF) < \tS_m(\CF_0;\CF_\IN) + N_m. $ Since (see for example (\ref{Eqnarray. tS/m^(n+1) limit = n vol }))
\begin{eqnarray}
\label{Eqnarray. N_m and tS(F0;F0)}
\mathop{\lim}_{m\to \infty} \frac{N_m}{m^{n}/n!} = \mult(\CF_0), \quad
\mathop{\lim}_{m\to \infty} \frac{\tS_m(\CF_0;\CF_0)}{m^{n+1}/(n+1)!} = n\cdot \mult(\CF_0), 
\end{eqnarray}
we see that the limit $S(\CF_0; \CF) := \lim_{m\to \infty} S_m(\CF_0; \CF)$ exists and $S(\CF_0; \CF) = S(\CF_0; \CF_\IN)$.

By definition, for any $a,b>0$, we have 
\begin{eqnarray*}
\lam_\max(a\CF_0;b\CF) = a^{-1}b\cdot\lam_\max(\CF_0;\CF), \quad
S(a\CF_0;b\CF) = a^{-1}b\cdot S(\CF_0;\CF). 
\end{eqnarray*} 

We define the {\it J-norm} of $\CF$ (with respect to $\CF_0$) by 
\begin{eqnarray*}
\BJ(\CF_0;\CF)  \,\,\, :=\,\,\, 
\lam_\max(\CF_0;\CF) - S(\CF_0;\CF) \ge 0. 
\end{eqnarray*}

\begin{lem}
If $\CF\seq \CF'$, then $\BF(\CF_0;\CF)\le \BF(\CF_0;\CF')$ for $\BF=\lam_\min,\lam_\max$ and $S$. 
\end{lem}

\begin{lem}
\label{Lemma. saturated fa seq fb, S(fa)=S(fb) implies fa=fb}
Let $x\in(X,\D)$ be a klt singularity, $\CF_0$ and $\CF\seq \CG$ be linearly bounded filtrations. If $S(\CF_0;\CF) = S(\CF_0;\CG)$, then $v(\CF)=v(\CG)$ for any $v\in\Val_{X,x}^*$.  

If we assume moreover that $\CF$ and $\CG$ are saturated, then $\CF=\CG$. 
\end{lem}

\begin{proof}
Otherwise, there exists $v\in\Val_{X,x}^*$ such that $1=v(\CF)>v(\CG)$. Then there exists $a\in\IZ_{\ge1}$ and $f\in\CG^a\setminus \CF^a$ such that $a>v(f)=:b$. We choose $c\in\IZ_{\ge1}$ such that $c(a-b)\ge1$. Consider the map $\phi: R\to \CG^{ack}/\CF^{ack}, g\mapsto g\cdot f^{ck}$. We have $\ker(\phi)\seq \CF^{k}_v$ since 
\begin{eqnarray*}
g\in\ker(\phi) &\Leftrightarrow& 
g\cdot f^{ck} \in\CF^{ack} \,\,\,\Rightarrow\,\,\,
v(g)+bck= v(g\cdot f^{ck}) \ge v(\CF^{ack})\ge ack \\
&\Rightarrow& v(g)\ge c(a-b)k\ge k 
\,\,\,\Rightarrow\,\,\, g\in \CF_v^k. 
\end{eqnarray*} 

We denote by $R_m:=R/\CF_{0}^{m}$ for any $m\in\IN$. Since $\CF_0$ and $\CF$ are linearly bounded by each other, there exists $d\in\IZ_{\ge 1}$ such that $\CF_0^{dm}\seq \CF^{am}$ for any $m\in\IN$. In particular, $\CG^j/\CF^j \cong \CG^j R_{dm}/\CF^j R_{dm}$ for any $j\le am$. Hence 
\begin{eqnarray*}
\tS_{cdm}(\CF_0;\CG) - \tS_{cdm}(\CF_0;\CF) 
&=& \sum_{j=1}^\infty \ell(\CG^j R_{cdm}/\CF^j R_{cdm})\\
&\ge&  \sum_{j=1}^{acm} \ell(\CG^j/\CF^j) 
\ge \sum_{k\in\IN, 1\le ack\le acm} \ell(\CG^{ack}/\CF^{ack}) \\
&\ge& \sum_{k=1}^{m} \ell(\CG^{ack}/\CF^{ack}) 
\ge \sum_{k=1}^{m} \ell(R/\CF_v^k) \,\,\,=:\,\,\, a_m. 
\end{eqnarray*} 
By Stolz's theorem, we have 
$$\mathop{\lim}_{m\to \infty}\frac{a_m}{m^{n+1}/(n+1)!} = \mathop{\lim}_{m\to \infty}\frac{\ell(R/\CF_v^m)}{m^{n}/n!} = \vol(v). $$
We set $b_m:=\tS_m(\CF_0;\CF_0) = m\ell(R_m)- \sum_{j=1}^m\ell(R_{j})$. Also by Stolz's theorem, we have
\begin{eqnarray}
\label{Eqnarray. tS/m^(n+1) limit = n vol }
\mathop{\lim}_{m\to \infty}\frac{b_m}{m^{n+1}/(n+1)!}
&=& \mathop{\lim}_{m\to \infty}\frac{m \ell(\CF_0^{m}/\CF_0^{m+1})}{m^{n}/n!} \\
\nonumber
&=& \mathop{\lim}_{m\to \infty} n\cdot\frac{\ell(\CF_0^{m}/\CF_0^{m+1})}{m^{n-1}/(n-1)!}  
=n\cdot \mult(\CF_0). 
\end{eqnarray} 
Hence
\begin{eqnarray*}
S(\CF_0;\CG) - S(\CF_0;\CF) 
\ge \mathop{\lim}_{m\to \infty}\frac{a_m}{b_{cdm}} 
= \frac{1}{n(cd)^{n+1}} \cdot \frac{\vol(v)}{\mult(\CF_0)} > 0. 
\end{eqnarray*} 
We get a contradiction. 

If $\CF$ and $\CG$ are saturated, then $\CF=\CG$ by \cite[Proposition 3.7]{BLQ22}. 
\end{proof}

\begin{cor}
\label{Corollary: J(v)=0}
If $\CF_0,\CF$ are saturated and $\BJ(\CF_0;\CF)=0$, then $\CF=a\CF_0$ for some $a>0$. 
\end{cor}

\begin{proof}
Rescale $\CF$ such that $\lam_{\max}(\CF_0;\CF)=1$. 
Then $\CF\seq \CF_0$ and 
\begin{eqnarray*}
0=\BJ(\CF_0;\CF)=1-S(\CF_0;\CF)=S(\CF_0;\CF_0)-S(\CF_0;\CF). 
\end{eqnarray*} 
Hence $\CF=\CF_0$ by Lemma \ref{Lemma. saturated fa seq fb, S(fa)=S(fb) implies fa=fb}. 
\end{proof}

Let $x\in(X,\D)$ be a klt singularity. For any effective $\IQ$-Cartier $\IQ$-divisor $D$ on $X$ such that $x\in \Supp(D)$, we have the following filtration determined by $D$: 
$$\CF_D^\lam = \{s\in R\mid \div(s)\ge \lam \cdot D\}. $$
Choose $r\in\IN$ such that $rD$ is Cartier. Then $rD=\div(f)$ for some $f\in R$, and $\CF_D^{mr}=f^m\cdot R$ is just the principle ideal generated by $f^m$. 

\begin{lem}
\label{Lemma. S-invariant of effective divisor}
Let $x\in(X,\D)$ be a klt singularity. For any $v\in\Val_{X,x}^*$ and any effective $\IQ$-Cartier $\IQ$-divisor $D$ on $X$, we have $S(v;\CF_D) = \frac{1}{n\cdot v(D)}$. 
\end{lem}

\begin{proof}
The proof follows from \cite[Proof of Lemma 3.3]{XZ22}. Let $\CF_f = \CF_{rD} = r^{-1}\CF_D$ be the filtration generated by $f$. Multiplying by $f^j$ induces an isomorphism $R/\CF_v^{m-j\cdot v(f)} \to \CF_f^j(R/\CF_v^m)$. Let $v'=\frac{v}{v(f)}$.  Then we have 
\begin{eqnarray*}
\tS_{m}(v'; \CF_f)
= \sum_{j=1}^{\lam_\max^{(m)}}\ell(\CF_f^j(R/\CF_{v'}^{m})) 
= \sum_{j=1}^{m}\ell( R/\CF_{v'}^{m-j}) = \sum_{j=1}^{m}\ell( R/\CF_{v'}^{j}). 
\end{eqnarray*} 
Hence by Stolz' theorem, we have 
\begin{eqnarray*}
\mathop{\lim}_{m\to \infty} \frac{\tS_{m}(v';\CF_f)}{m^{n+1}/(n+1)!} 
= \mathop{\lim}_{m\to \infty} \frac{\ell(R/\CF_{v'}^m)}{m^{n}/n!} 
=\vol(v'). 
\end{eqnarray*} 
On the other hand, by (\ref{Eqnarray. tS/m^(n+1) limit = n vol }), we have 
\begin{eqnarray*}
\mathop{\lim}_{m\to \infty} \frac{\tS_{m}(v';v')}{m^{n+1}/(n+1)!} 
=n\cdot \vol(v'). 
\end{eqnarray*} 
We conclude that $S(v';\CF_f) = \frac{1}{n}$, hence $S(v;\CF_D)=\frac{r}{v(f)}S(v';\CF_f)=\frac{1}{n\cdot v(D)}$. 
\end{proof}

\subsection{Test configurations}

Following \cite[Definition 2.14]{LWX18}, we define the test configurations of a klt singularity $x\in(X,\D)$. 
We remark that the definition of test configurations (and the corresponding filtrations) of a log Fano cone $(X,\D,\xi_0)$ in \cite[Definition 2.14]{LWX18} is independent of the choice of $\xi_0\in\Reeb$. But the definition of generalized Futaki invariants depend on $\xi_0\in\Reeb$. 

\begin{defi}[Test configurations] \rm 
Let \(x\in (X, \D)\) be a klt singularity with a good \(\IT\)-action. A {\it \(\IT\)-equivariant $\IQ$-Gorenstein normal test configuration} (or simply called a normal test configuration) of \((X, \D)\) is a triple \((\mathcal{X}, \D_\CX; \eta)\) with a map \(\pi : (\mathcal{X}, \D_\CX) \to \IA^1_t\) satisfying the following conditions:  

\begin{enumerate}
    \item \(\mathcal{X}=\Spec(\CR)\) is an affine normal variety, \(\pi : \mathcal{X} \to \IA^1\) is a flat family, and \(\D_\CX\) is a $\IQ$-divisor on \(\mathcal{X}\) with \(\operatorname{Supp}(\D_\CX)\) not containing any component of a fiber of \(\pi\) such that $K_\CX+\D_\CX$ is $\IQ$-Cartier;  
    \item \(\eta\) is an algebraic holomorphic vector field that generates a \(\IG_m\)-action on \((\mathcal{X}, \D_\CX)\) such that \(\pi\) is \(\IG_m\)-equivariant and \(\pi_\ast \eta = -t\partial_t\). As a consequence, there is an isomorphism \(\phi : (\mathcal{X}, \D_\CX) \times_{\IA^1\setminus \{0\}} (\IA^1\setminus \{0\}) \cong (X, \D) \times (\IA^1\setminus \{0\})\);  
    \item The torus \(\IT\) acts on \((\mathcal{X}, \D)\) fiberwise and commutes with the \(\IG_m\)-action generated by \(\eta\), and coincides with the action on the first factor when restricted to \((\mathcal{X}, \D_\CX) \times_{\IA^1\setminus \{0\}} (\IA^1\setminus \{0\}) \cong (X, \D) \times (\IA^1\setminus \{0\})\).
\end{enumerate}
\end{defi}

A normal test configuration is called {\it (weakly) special} if moreover $(\CX,\D_\CX+\CX_0)$ is plt (lc). A test configuration $(\CX, \D_\CX;\eta)$ is of {\it product type} if there is a \( \IT \)-equivariant isomorphism \( (\mathcal{X}, \D_\CX) \cong (X, \D) \times \IA^1 \) and \( \eta = \eta_0 - t\partial_t \) where \( \eta_0 \) is a coweight vector of \( \IT \) and \( t\partial_t \) is the canonical lifting of \( t\partial_t \) on \( \IA^1 \) through the second projection. In this case, we will denote \( (\mathcal{X}, \D_\CX; \eta) \) by  
\[
(X \times \IA^1, \D \times \IA^1; \eta) =: (X_{\IA^1}, \D_{\IA^1}; \eta).
\]

\begin{rmk}\rm
The one-parameter subgroup $\eta_t$ moves points in $\CX\setminus \CX_0$ away from $\CX_0$ as $t\to 0$. For any $\xi_0$ in the Reeb cone of $R$, we see that $a\xi_0+\eta$ lies in the Reeb cone of $\CR_0$ with respect to the $(\IT\times \IG_m)$-action when $a\gg0$. 
\end{rmk}


For any test configurations $(\CX, \D_\CX;\eta)$, we have the following $\IN$-filtration $\CF=\CF_{(\CX, \D_\CX;\eta)}$:
\begin{eqnarray*}
\CF^\lam = \{ f\in R \mid t^{-\lam} \bar{f} \in \CR \}, 
\end{eqnarray*} 
where $\bar{f} \in R[t,t^{-1}]$ is the $\IG_m$-invariant rational section of $\CO_\CX$ determined by $f\in R$. By Rees construction \cite[Lemma 2.17]{LWX18}, we have the following canonical isomorphism of $\Ik[t]$-algebra:
\begin{eqnarray*}
\CR \cong \bigoplus_{\lam\in\IZ} t^{-\lam}\cdot\CF^\lam. 
\end{eqnarray*} 
Moreover, by \cite[Lemma 2.21]{LWX18}, we have the following characterization of the filtration $\CF$. Since $\CR$ is finitely generated, so is $\CF$. There exists $m\in \IZ_{>0}$ such that $\CF^{ml}=(\CF^m)^l$ (as ideals) for any $l\in\IN$. Let $\mu: Y\to X$ be the normalized blowup of $\CF^m\seq R$. Then there exists an effective anti-ample Cartier divisor $mE$ on $Y$ such that $\CF^m\cdot\CO_Y = \CO_Y(-mE)$. We denote by $E=\sum_j b_j E_j$ for some prime divisors $E_j$ and $b_j\in\IQ_{>0}$. Then 
\begin{eqnarray*}
\CF^{ml} 
\,\,\,=\,\,\, \mu_* \CO_Y(-mlE) 
&=& \{f\in R \mid \mu^*\div(f)\ge mlE\} \\
&=& \{f\in R \mid \ord_{E_j}(f)\ge mlb_j,  \forall j\}
\,\,\,=\,\,\, \bigcap_{j} \fa_{ml}(v_{E_j})
\end{eqnarray*} 
for any $l\in \IN$, where $v_{E_j} = b_j^{-1}\ord_{E_j}$. Hence $\CF$ is saturated by \cite[Lemma 3.19]{BLQ22}. 


\subsection{Log Fano cone singularities}

\begin{defi}[Good actions]\rm 
Let $(X,\D)$ be an affine normal pair of dimension $n$, and $\IT=\IG_m^r$. A $\IT$-action on $(X,\D)$ is called {\it good} if it is effective and there is a closed $\IT$-fixed point $x$ which is in the closure of any $\IT$-orbit. We will simply say that the singularity $x\in(X,\D)$ admits a good $\IT$-action. 
\end{defi}

Denote by $X = \Spec(R)$ for some $\Ik$-algebra $R$. Then we have weight decomposition $R=\oplus_{\alpha \in M} R_\alpha$ by the $\IT$-action, where $M=\Hom(\IT, \IG_m)\cong \IZ^r$ is the weight lattice. We denote by $N=M^\vee$ the coweight lattice. Note that $R_0 = \Ik$ since the $\IT$-action is good. 
We define the Reeb cone $\Bt_\IR^+\seq N_\IR$ by 
\begin{eqnarray*}
\Bt_\IR^+:=\{\xi\in N_\IR\mid \la \alpha,\xi\ra >0,\,\, \forall \alpha, R_{\alpha}\ne 0 \}. 
\end{eqnarray*}
For any $\xi_0\in \Bt_\IR^+$, The triple $(X,\D,\xi_0)$ is called a {\it log Fano cone singularity} polarized by $\xi_0$. 

The vectors in the Reeb cone give some special valuations over $x\in X$.  
For any $\xi\in \Bt_\IR^+$, we define the {\it toric valuation} $\wt_\xi$ by setting $\wt_\xi(f) = \la\alpha,\xi\ra$ for any $f\in R_\alpha$. 
It's clear that $\wt_\xi \in \Val^\IT_{X,x}$ if and only if $\xi\in \Bt_\IR^+$. It's clear that $\wt_\xi, \xi\in\Reeb$ are special valuations since they induce product type (multi-step) special degenerations (\cite[Theorem 4.1]{XZ22}). 

For any valuation $v: K(X)^\times\to \IR$ and any $\xi\in N_{\IR}$, we define the {\it $\xi$-twist} $v_\xi$ of $v$ by setting $v_\xi(f)=v(f)+\la\alpha,\xi\ra$ for any $f\in R_\alpha$. 

\begin{lem}\cite[Lemma 2.7]{LW24}
\label{Lemma. A(v_xi)=A(v)+A(xi)}
The function $A(\xi) := A_{X,\D}(\wt_\xi)$ is linear on $\Bt_\IR^+$. For any $\xi\in \Bt_\IR^+$ and $v\in \Val^{\IT}_X$, we have 
$A_{X,\D}(v_\xi) = A_{X,\D}(v) + A(\xi). $
\end{lem}

\begin{lem}
\label{Lemma. twist valuations to vertex}
For any $\xi\in \Bt_\IR^+$ and $v\in \Val^{\IT}_X$, we have $v_{\xi} \in \Val^{\IT}_{X,x}$. 
\end{lem}
\begin{proof}
Since the $\IT$-action is good, we have $\fm_x=\oplus_{\alpha\ne 0}R_\alpha$. For any $0\ne f\in R_\alpha$, we have $v_{\xi}(f) =v(f) +\la\alpha,\xi\ra > 0$. Assume that $\fm_x$ is generated by $f_1,\cdots,f_N$, where $f_i\in R_{\alpha_i}$. Then any $f\in \fm_x$ is of the form $f=\sum_i g_i f_i$ for some $g_i\in R$. Hence $v_{\xi}(f)\ge \min_i\{v_{\xi}(g_i)+v_{\xi}(f_i) \} \ge \min_i\{v_{\xi}(f_i) \}$. We see that $v_{\xi}(\fm_x) = \min_i\{v_{\xi}(f_i) \} >0$. 
\end{proof}

\begin{rmk}\rm
\label{Remark. simplify notations CF_triv,xi0 = CF_wt_xi0}
For simplicity of notations, we will use, for example,  $\vol(\xi_0), \lam_{\max}(\xi_0;\CF)$ and $S(\xi_0;\xi)$ instead of $\vol(\wt_{\xi_0}), \lam_\max(\CF_{\wt_{\xi_0}};\CF)$ and $S(\CF_{\wt_{\xi_0}};\CF_{\wt_\xi})$ respectively. 
\end{rmk}




Let $\CF$ be a $\IT$-invariant filtration on $R$, which means that $\CF^\lam = \oplus_{\alpha\in M} \CF^\lam R_\alpha$ (where $\CF^\lam R_\alpha := \CF^\lam\cap R_\alpha$). 
For any $\xi\in \Reeb$, the $\xi$-{\it twist} $\CF_\xi$ of $\CF$ is defined by 
\begin{eqnarray*}
\CF_{\xi}^\lam := \bigoplus_{\alpha\in M} \CF^{\lam-\la\alpha,\xi\ra} R_{\alpha}, \quad \lam\in \IR_{\ge0}.  
\end{eqnarray*}
Note that $\CF_{\xi}^\lam$ is an ideal since $(\CF^{\lam-\la\alpha,\xi\ra} R_{\alpha}) \cdot R_\beta \seq \CF^{\lam-\la\alpha+\beta,\xi\ra} R_{\alpha+\beta}$ for any $\alpha,\beta \in M$, and $\CF^\lam\seq \CF_\xi^\lam$. Hence $\CF_\xi$ is an ($\fm_x$-)filtration. 

Let $\CG$ be another $\IT$-invariant filtration. Then for any $\xi\in\Reeb$ we have $\CF_\xi\seq \CG_\xi$ if and only if $\CF\seq \CG$. Indeed, the `if' part is clear since $\CF^{\lam-\la\alpha,\xi\ra} R_{\alpha} \seq \CG^{\lam-\la\alpha,\xi\ra} R_{\alpha}$ for any $\lam\in\IR_{\ge0}$ and $\alpha\in M$. For the `only if' part, the inclusion $\CF^\lam R_\alpha \seq \CG^\lam R_\alpha$ follows from $\CF^{\lam'-\la\alpha,\xi\ra} R_{\alpha} \seq \CG^{\lam'-\la\alpha,\xi\ra} R_{\alpha}$ where $\lam'=\lam+\la\alpha,\xi\ra$. 


\begin{lem}
\label{Lemma: twist of valuations and filtrations}
For any $v\in \Val^\IT_{X,x}$ and $\xi \in \Bt_\IR^+$, we have $\CF_{v,\xi} = \CF_{v_\xi}$. 
\end{lem}
\begin{proof}
It suffices to prove $\CF_{v,\xi}^\lam R_\alpha = \CF^\lam_{v_\xi} R_\alpha$ for any $\lam\in\IR$ and $\alpha\in M$. Since $v_\xi(s)=v(s)+\la\alpha,\xi\ra$ for any $s\in R_\alpha$. We conclude that 
\begin{eqnarray*}
s \in \CF_{v,\xi}^\lam R_\alpha = \CF_{v}^{\lam-\la\alpha, \xi\ra} R_\alpha
\Leftrightarrow 
v(s)\ge \lam-\la\alpha,\xi\ra 
\Leftrightarrow 
v_\xi(s) \ge \lam
\Leftrightarrow
s \in \CF^\lam_{v_\xi} R_\alpha. 
\end{eqnarray*}
\end{proof}

\begin{lem}
For any $\IT$-invariant filtration $\CF$ and $\xi\in\Bt_\IR^+$, we have
$\lct(\CF_\xi) = \lct(\CF)+A(\xi)$. 
\end{lem}
\begin{proof}
Let $v$ be a $\IT$-invariant minimizer of $\lct(\CF)$ such that $v(\CF)=1$. Then $\CF\seq \CF_v$ and 
\begin{eqnarray*}
A_{X,\D}(v)=\lct(\CF)\le\lct(\CF_v)\le A_{X,\D}(v). 
\end{eqnarray*}
Hence $\CF_\xi\seq \CF_{v,\xi} = \CF_{v_\xi}$ and $\lct(\CF_\xi)\le A_{X,\D}(v_\xi)=\lct(\CF)+A(\xi)$. 

Conversely, let $w$ be a $\IT$-invariant minimizer of $\lct(\CF_\xi)$ with $w(\CF_\xi)=1$. With the same argument as above, we have $\CF_\xi\seq \CF_w$ and $\lct(\CF_\xi)=\lct(\CF_w)=A_{X,\D}(w)$. Hence $\CF\seq \CF_{w_{-\xi}}$. Note that $w_{-\xi}$ is a valuation on $K(X)$, $w_{-\xi}(1)=0$ and $\CF^\lam\seq\CF^\lam(w_{-\xi})$ are $\fm_x$-primary. Hence $w_{-\xi}\in \Val^\IT_{X,x}$ and $(w_{-\xi})_\xi =w$. So $A_{X,\D}(w_{-\xi})+A(\xi)=A_{X,\D}(w)$. We conclude that 
$$\lct(\CF)\le \lct(w_{-\xi}) 
\le  A_{X,\D}(w_{-\xi}) 
= A_{X,\D}(w) - A(\xi)
=\lct(\CF_\xi)-A(\xi).$$
\end{proof}


\begin{lem}
\label{Lemma: S and T, twist by xi_0}
For any $\IT$-invariant linearly bounded filtration $\CF$ and $a\in \IR_{>0}$, we have 
\begin{eqnarray}
\lam_\max(\xi_0;\CF_{a\xi_0}) = \lam_\max(\xi_0;\CF) + a. 
\end{eqnarray}
\end{lem}
\begin{proof}
Let $\CF_0=\CF_{\wt_{\xi_0}}$. For any $\lam>\lam_\max(\xi_0;\CF) + a$ and $m\in\IN$, we have $\CF^{(\lam-a)m}\seq \CF_0^m$. In other words, for any $\alpha$ with $\la\alpha,\xi_0\ra<m$, we have $\CF^{(\lam-a)m}R_\alpha=0$. Hence $\CF_{a\xi_0}^{\lam m} R_\alpha= \CF^{\lam m-a\la\alpha,\xi_0\ra}R_\alpha \seq \CF^{(\lam-a)m}R_\alpha =0$, that is, $\CF^{\lam m}_{a\xi_0} \seq \CF_0^m$. We get `$\le$'. 

Conversely, for each $\alpha \in M$, we define 
\begin{eqnarray}
\label{Eqnarray. lam_max^alpha}
\lam_\max^{(\alpha)}(\CF) := \max\{\lam>0 \mid \CF^\lam R_\alpha \ne 0\}. 
\end{eqnarray}
Then $\lam_\max^{(\alpha)}(\CF_{a\xi_0}) = \lam_\max^{(\alpha)}(\CF) + a\la\alpha,\xi_0 \ra$ and $\lam^{(m)}_\max(\xi_0;\CF) = \max\{\lam_\max^{(\alpha)}(\CF) \mid m-1 \le \la\alpha, \xi_0\ra <m \}. $ Hence 
\begin{eqnarray*}
\frac{1}{m} \lam_\max^{(m)}(\xi_0; \CF_{a\xi_0}) \ge \frac{1}{m}\lam_\max^{(m)}(\xi_0; \CF) +a(1-\frac{1}{m}). 
\end{eqnarray*}
We conclude by (\ref{Eqnarray. lam_max^(m) convergence}) and taking $m\to \infty$. 
\end{proof}

\begin{lem}
\label{Lemma. S(F_xi)=S(F)+S(xi)}
For any $\IT$-invariant linearly bounded filtration $\CF$ and $\xi\in\Reeb$, we have 
\begin{eqnarray}
S(\xi_0;\CF_\xi) = S(\xi_0;\CF) + S(\xi_0;\xi). 
\end{eqnarray}
\end{lem}
Note that $S(\xi_0;\xi_0)=1$, we have $S(\xi_0;\CF_{a\xi_0}) = S(\xi_0;\CF)+a$ for any $a>0$. 
\begin{proof}
Let $\CF_0=\CF_{\wt_{\xi_0}}$ and $L_m = \ell(\CF_0^{m}/\CF_0^{m+1})$. We have 
\begin{eqnarray*}
S'_m(\xi_0;\CF_{\xi}) 
&=& \frac{1}{L_m} \sum_\lam \frac{\lam}{m} \sum_{m\le\la\alpha, \xi_0\ra < m+1} \ell( \Gr_{\CF_\xi}^\lam R_\alpha) 
\,\,\,=\,\,\, \frac{1}{L_m} \sum_\lam \frac{\lam}{m} \sum_{m\le\la\alpha, \xi_0\ra < m+1} \ell( \Gr_{\CF}^{\lam - \la\alpha, \xi\ra } R_\alpha) \\
&=& \frac{1}{L_m} \sum_{m\le\la\alpha, \xi_0\ra < m+1} \sum_\lam \frac{\lam + \la\alpha, \xi\ra }{m} \ell( \Gr_{\CF}^{\lam} R_\alpha) \\
&=& S'_m(\xi_0;\CF) + \frac{1}{L_m} \sum_{m\le\la\alpha, \xi_0\ra < m+1} \frac{\la\alpha, \xi\ra }{m} \ell(R_\alpha)
\,\,\,=\,\,\,S'_m(\xi_0;\CF) +S'_m(\xi_0;\xi) . 
\end{eqnarray*}
We conclude by taking $m\to \infty$. 
\end{proof}

\begin{lem}
\label{Lemma. lam_min xi_0 twist}
For any $v\in \Val_{X,x}^{\IT,*}$ and $a> -\lam_\min(\xi_0;v)$, we have $v_{a\xi_0} \in \Val_{X,x}^{\IT,*}$ and 
\begin{eqnarray}
\label{Eqnarray. lam_min xi_0 twist}
\lam_\min(\xi_0;v_{a\xi_0}) = \lam_\min(\xi_0;v) + a. 
\end{eqnarray}
\end{lem}
\begin{proof}
For any $\alpha$, we have $v_{a\xi_0}(R_\alpha) = v(R_\alpha) + a\cdot \la\alpha,\xi_0\ra$. Hence $v_{a\xi_0}(\CF_{\wt_{\xi_0}}) = v(\CF_{\wt_{\xi_0}}) + a$ by restricting $1-\frac{1}{m}\le \frac{1}{m}\la\alpha,\xi_0\ra < 1$ and letting $m\to \infty$. This is just (\ref{Eqnarray. lam_min xi_0 twist}). Recall that a valuation $w$ on $X$ is supported at $x$ if and only if $\lam_\min(\xi_0;w) (=w(\CF_{\wt_{\xi_0}}) ) > 0$. Hence $v_{a\xi_0} \in \Val_{X,x}^{\IT,*}$ since $\lam_\min(\xi_0;v_{a\xi_0}) = \lam_\min(\xi_0;v) + a >0$. 
\end{proof}

\begin{defi}[Geodesics]\rm
Let $\{\CF_i\}_{1\le i\le r}$ be linearly bounded filtrations on $R$. For any $\alpha = (\alpha_1,\cdots,\alpha_r) \in \IR_{\ge 0}^r$, we define the filtration $\CF_\alpha$ by 
\begin{eqnarray}
\label{Eqnarray. Geodesic filtration}
\CF_{\alpha}^\lam &:=& 
\sum_{\alpha_1\lam_1+\cdots+\alpha_r\lam_r \ge \lam}
\CF_1^{\lam_1} \cap \cdots \cap  \CF_r^{\lam_r} \\
&=&\{f\in R \mid \alpha_1 \cdot \ord_{\CF_1}(f) +\cdots+ \alpha_r \cdot \ord_{\CF_r}(f) \ge \lam \}. 
\end{eqnarray}
The collection of filtrations $\{\CF_\alpha\}_{\alpha_1+\cdots+\alpha_r = 1}$ is called the {\it geodesic} connecting $\CF_1,\cdots,\CF_r$. We have 
\begin{eqnarray}
\label{Eqnarray. ord_F is linear along geodesic}
\ord_{\CF_\alpha} = \alpha_1 \ord_{\CF_1} + \cdots + \alpha_{r}\ord_{\CF_r}. 
\end{eqnarray}
\end{defi}

In the case $r=2$, let $\CF_0,\CF_1$ be $\IT$-invariant linearly bounded filtrations. We set $\alpha_0=1-t$ and $\alpha_1=t$ for some $0\le t\le1$, and denote the geodesic $\CF_{(\alpha_0,\alpha_1)}$ connecting $\CF_0,\CF_1$ by $\CF_t$. By \cite[Lemma 3.1]{AZ22}, there exists a $\IT$-invariant basis $\{f_i\}$ of the $\Ik$-vector space $R/\CF_{\wt_{\xi_0}}^{m} = \oplus_{\la\alpha,\xi_0\ra< m} R_\alpha$ that is compatible with both $\CF_0$ and $\CF_1$. Then we have 
\begin{eqnarray*}
\tS_m(\xi_0;\CF_t) = \sum_i \ord_{\CF_t}(f_i) = \sum_i (1-t)\ord_{\CF_0}(f_i) +t \ord_{\CF_1}(f_i) = (1-t)\tS_m(\xi_0; \CF_0) + t\tS_m(\xi_0; \CF_1). 
\end{eqnarray*}
By multiplying with $\tS_m(\xi_0;\xi_0)^{-1}$ and taking $m\to \infty$, we get
\begin{eqnarray}
\label{Eqnarray. S is linear along geodesics}
S(\xi_0;\CF_t)  = (1-t)S(\xi_0; \CF_0) + t S(\xi_0; \CF_1). 
\end{eqnarray}
In other words, $S(\xi_0;-)$ is linear along geodesics.

\begin{defi}[Quasi-monomial valuations]\rm
\label{Definition. Quasi-monomial valuations}
Let $(Y,E=E_1+\cdots+E_r)\to X$ be a log smooth model and $\alpha=(\alpha_1,\cdots, \alpha_r) \in \IR_{\ge 0}^r$, we define the valuation $v_\alpha$ by 
\begin{eqnarray}
\label{Eqnarray. quasi-monomial valuations}
v_\alpha(f) = \min\{ \sum_{1\le i\le r} \alpha_i \beta_i\mid f_\beta\ne 0 \}
\end{eqnarray}
for any $f = \sum_\beta f_\beta z^\beta \in \hat{\CO}_{X,\eta}$, where $\eta$ is the generic point of some irreducible component of $\cap_{1\le i\le r} E_i$ and $z_1,\cdots, z_r \in \hat{\CO}_{X,\eta}$ are local functions such that $E_i = (z_i=0)$. The valuation $v_\alpha$ is called a {\it quasi-monomial valuation} over $X$. We denote by 
\begin{eqnarray}
\QM_\eta(Y,E) = \{v_\alpha\in \Val_{X}\mid \alpha \in \IR_{\ge 0}^r \}. 
\end{eqnarray}

Instead of assuming that $(Y,E)$ is SNC at $\eta$, we may generalize the assumption and assume that $(Y,E)$ is toroidal at $\eta$, that is, there is an semi-local SNC pair $(\eta'\in Y', E_1'+\cdots+E_r')$ with an abelian group $G$ action such that 
\begin{eqnarray}
\label{Eqnarray. toroidal quotient}
(\eta'\in Y', E'= E_1'+\cdots+E_r')/G \cong (\eta\in Y, E= E_1+\cdots+E_r). 
\end{eqnarray}
Assume that the pull-back of $E_i\seq Y$ to $Y'$ is $n_iE_i'$ (hence $\ord_{E_i} = \frac{1}{n_i}\ord_{E_i'}|_Y$). Then for any $\alpha = (\alpha_1,\cdots,\alpha_r) \in \IR_{\ge 0}^r$, we denote by $\alpha'=(\frac{\alpha_1}{n_1},\cdots,\frac{\alpha_r}{n_r})$. We define $v_\alpha = v_{Y,\alpha} = v_{Y', \alpha'}|_{K(Y)}$ by the restriction of $v_{Y',\alpha'}$ (defined by (\ref{Eqnarray. quasi-monomial valuations})) on $K(Y)$ (via the inclusion $K(Y)\cong K(Y')^G \seq K(Y')$). Hence 
\begin{eqnarray*}
v_{Y,\alpha}(E_i) = v_{Y',\alpha'}(n_i E_i') = n_i \cdot \frac{\alpha_i}{n_i} = \alpha_i. 
\end{eqnarray*}
We denote by $\QM_\eta(Y,E) = \{v_\alpha \in \Val_Y\mid \alpha \in \IR_{\ge 0}^r\}$. Note that we have the following isomorphism induced by the isomorphism (\ref{Eqnarray. toroidal quotient}): 
\begin{eqnarray*}
\IR_{\ge 0}^r \cong \QM_\eta(Y,E) 
&\to& \QM_{\eta'}(Y',E') \cong \IR_{\ge 0}^r \\
\alpha = (\alpha_1,\cdots,\alpha_r) 
&\mapsto& (\frac{\alpha_1}{n_1},\cdots,\frac{\alpha_r}{n_r}) = \alpha'. 
\end{eqnarray*}
\end{defi}

\begin{lem}
\label{Lemma. geodesic seq line-sigment}
Let $\CF_{i}=\CF_{\ord_{E_i}}$. Then for any $\alpha = (\alpha_1,\cdots,\alpha_r) \in \IR_{\ge 0}^r$, we have $\CF_\alpha \seq \CF_{v_\alpha}$. 
\end{lem}
\begin{proof}
For any $f\in R$, write $f = \sum_\beta f_\beta z^\beta \in \hat{\CO}_{Y',\eta'}$ as a local function at $\eta'\in Y'$. For any $\alpha = (\alpha_1,\cdots,\alpha_r) \in \IR_{\ge 0}^r$, we see that $f \in \CF_\alpha^\lam $ if and only if $\sum_i \alpha_i \beta^{(i)} \ge \lam$, where $\beta^{(i)} = \ord_{E_i}(f) = \frac{1}{n_i} \ord_{E_i'}(f) = \frac{1}{n_i}\min\{\beta_j\mid f_\beta \ne 0, \beta=(\beta_1,\cdots,\beta_r)\in \IN^r\}$. On the other hand, by definition of quasi-monomial valuations, $f\in \CF_{v_\alpha}^\lam$ if and only if $\sum_i \frac{\alpha_i}{n_i}\cdot  \beta_i \ge \lam$ for any $\beta=(\beta_1,\cdots,\beta_r)\in\IN^r$ satisfying $f_\beta\ne 0$. In particular, we have $\CF_\alpha \seq \CF_{v_\alpha}$. 
\end{proof}

\begin{cor}
\label{Corollary. S is concave on quasi-monomial cone}
Let $\sigma\seq \Val_{X,x}$ be a $\IT$-invariant quasi-monomial simplicial cone, then the function $\sigma\to \IR_{\ge 0}, \,\, v\mapsto S(\xi_0; v)$ is concave. 
\end{cor}

\begin{proof}
For any $v_0,v_1\in \sigma_\IQ$, let $\CF_0=\CF_{v_0}, \CF_1=\CF_{v_1}$. Let $v_t\in \sigma$ be the linear combination $(1-t)v_0+tv_1$ in $\sigma$, and $\CF_t$ be the geodesic connecting $\CF_0,\CF_1$. Then $\CF_t\seq \CF_{v_t}$ by Lemma \ref{Lemma. geodesic seq line-sigment}. Hence by (\ref{Eqnarray. S is linear along geodesics}), we have $S(\xi_0; \CF_{v_t}) \ge S(\xi_0; \CF_t) = (1-t)S(\xi_0; \CF_{v_0}) + t S(\xi_0; \CF_{v_1}). $
\end{proof}

We will see in Theorem \ref{Theorem. Special cone: geodesic = line sigment} that the inclusion $\CF_\alpha \seq \CF_{v_\alpha}$ in Lemma \ref{Lemma. geodesic seq line-sigment} is an equality if $E_1,\cdots,E_r \in \Val_{X,x}$ spans a special quasi-monomial cone $\sigma$. In this case, the function $S(\xi_0; -)$ is linear on $\sigma$.


\begin{lem}
\label{Lemma. lam_max is continuous on quasi-monomial cone}
Let $\sigma\seq \Val_{X,x}^*$ be a $\IT$-invariant quasi-monomial simplicial cone, then the function $\sigma\to \IR_{\ge 0}, \,\, v\mapsto \lam_\max(\xi_0; v)$ is continuous. 
\end{lem}
\begin{proof}
Since $\lam_\max(\xi_0; v)$ is homogeneous of degree one in $v$, it suffices to show that $v\mapsto \lam_\max(\xi_0; v)$ is continuous on the simplex $D=\{(x_1,\cdots,x_r)\mid x_1+\cdots + x_r = 1\} \seq \IR^{r}_{\ge 0}\cong \sigma$. 

Fix $v\in D$ and $\lam < \lam_\max(\xi_0; v)$. By definition, there exists $m\in \IN$ and $f\in \CF_v^{m\lam} \setminus \CF_{\wt_{\xi_0}}^m$. By \cite[Theorem 1]{BFJ14}, there exists a constant $A_0>0$ such that for any $g\in R$, the function $D\to \IR_{\ge 0}: w\mapsto w(g)$ is Lipschitz continuous (with respect to the Euclidean norm $|\cdot |$ on $D$) with Lipschitz constant at most $A_0\ord_x(g)$. Since $\CF_{\wt_{\xi_0}}$ is linearly bounded, there exists a constant $A>0$ such that for any $g\in R$, we have $A_0\ord_x(g) \le A \wt_{\xi_0}(g)$. Since $A \wt_{\xi_0}(f)< Am$, we have $|w(f)-v(f)| \le Am|w-v|$. Then
\begin{eqnarray*}
 w(f)\ge v(f) - Am|w-v| \ge (\lam - A|w-v|) m. 
\end{eqnarray*}
In particular, $\lam_\max(\xi_0; w) \ge \lam - A|w-v|$. Letting $\lam \to \lam_\max(\xi_0; v)$ we get 
\begin{eqnarray*}
\lam_\max(\xi_0; w)-\lam_\max(\xi_0; v) \ge - A|w-v|. 
\end{eqnarray*}
Exchanging the roles of $v$ and $w$ we conclude that 
\begin{eqnarray*}
|\lam_\max(\xi_0; w)-\lam_\max(\xi_0; v)| \le A|w-v|. 
\end{eqnarray*}
\end{proof}

\section{Polystability}
\label{Section. Polystability}

In this section, we recall various polystability notions of a log Fano cone and show that they are all equivalent.

\subsection{Definition of K/Ding-stability}

Let $(X=\Spec(R),\D,\xi_0)$ be a log Fano cone singularity with a good $\IT$-action. For any $\IT$-invariant linearly bounded filtration $\CF$, we define the {\it Ding invariant} by 
\begin{eqnarray}
\BD_{X,\D,\xi_0}(\CF) = \BD(\CF) := \lct(X,\D;\CF) - A(\xi_0)S(\xi_0;\CF). 
\end{eqnarray}

Following \cite{LWX18}, for any $\IT$-equivariant test configuration $(\CX, \D_\CX,\xi_0;\eta)$ of $(X,\D,\xi_0)$, we define the {\it generalized Futaki invariant} by 
\begin{eqnarray}
\Fut(\CX, \D_\CX,\xi_0;\eta) = \frac{D_{T_{\xi_0}(\eta)}\vol_{\CX_0}(\xi_0)}{\vol_{\CX_0}(\xi_0)}, 
\end{eqnarray}
where $T_{\xi_0}(\eta)=\frac{1}{n}(A(\xi_0)\eta - A(\eta)\xi_0)$. We also define the corresponding Ding invariant by 
\begin{eqnarray}
\BD(\CX, \D_\CX,\xi_0;\eta) &:=& \BD_{X,\D,\xi_0}(\CF_{(\CX, \D_\CX;\eta)}). 
\end{eqnarray}

\begin{defi}[K/Ding-semistability] \rm
\label{Definition. Ding-stability for filtrations}
Let $(X,\D,\xi_0)$ be a log Fano cone. It is called {\it $\IT$-equivariantly Ding-semistable}  for filtrations (resp. normal test configurations, special test configurations) if $\BD(\CF)\ge 0$ for any $\IT$-invariant linearly bounded filtration $\CF$ (resp. any filtration $\CF$ of $\IT$-equivariant normal test configuration, any filtration $\CF$ of $\IT$-equivariant special test configuration). 

Replacing $\BD$ by $\Fut$, we get the definition of {\it $\IT$-equivariantly K-semistability} for normal test configurations (resp. special test configurations). 
\end{defi}

\begin{rmk}\rm
We will see that Ding-semistability for filtrations has a good characterization by the delta invariant in Theorem \ref{Theorem. D-ss equiv. delta >= 1}, which can be seen as ``stability for valuations''. We simply say that a log Fano cone is {\it Ding-semistable} if it is $\IT$-equivariantly Ding-semistable for filtrations. 
\end{rmk}

\begin{defi}[K/Ding-polystability]\rm
\label{Definition. K/Ding polystability}
A Ding-semistable log Fano cone $(X,\D,\xi_0)$ is called {\it $\IT$-equivariantly Ding-polystable} for normal test configurations (resp. special test configurations) if for any $\IT$-equivariantly normal test configuration (resp. special test configuration) $(\CX, \D_\CX,\xi_0;\eta)$, $\BD(\CX, \D_\CX,\xi_0;\eta)= 0$ implies that $(\CX, \D_\CX,\xi_0;\eta)$ is of product type. 

Replacing $\BD$ by $\Fut$, we get the definition of {\it $\IT$-equivariantly K-polystability} for normal test configurations (resp. special test configurations). 
\end{defi}

Following from \cite{XZ20}, we define the delta invariant of the log Fano cone $(X,\D,\xi_0)$ by 
\begin{eqnarray}
\delta_{\IT}(X,\D,\xi_0) := \mathop{\inf}_{v\in \Val^{\IT,*}_{X,x}} \frac{A_{X,\D}(v)}{A(\xi_0)S(\xi_0;v)}. 
\end{eqnarray} 

\begin{lem}
\label{Lemma: delta le 1}
For any $\xi_0\in \Bt^+_\IR$, we have $\delta_\IT(X,\D,\xi_0)\le1$. 
\end{lem}
\begin{proof}
For any valuation $v\in\Val^{\IT,*}_{X,x}$, by Lemma \ref{Lemma. A(v_xi)=A(v)+A(xi)} and Lemma \ref{Lemma. S(F_xi)=S(F)+S(xi)}, we have 
\begin{eqnarray*}
\frac{A_{X,\D}(v_{a\xi_0})}{A(\xi_0)S(\xi_0;v_{a\xi_0})} 
= 
\frac{A_{X,\D}(v) + aA(\xi_0)}{A(\xi_0)S(\xi_0;v) + aA(\xi_0)} 
\to 1, 
\end{eqnarray*} 
when $a\to +\infty$. 
\end{proof}

By the work of \cite{LW24}, we know that $\delta_\IT(X,\D,\xi_0)=1$ is equivalent that $(X,\D,\xi_0)$ is Ding-semistable for NA-functionals. For the sake of completeness, we present a proof in our notations. 

\begin{thm}
\label{Theorem. D-ss equiv. delta >= 1}
Let $(X,\D,\xi_0)$ be a log Fano cone. It is $\IT$-equivariantly Ding-semistable for filtrations if and only if $\delta_\IT(X,\D,\xi_0)\ge 1$. 
\end{thm}
\begin{proof}
For any $v\in \Val^{\IT,*}_{X,x}$, we have $A_{X,\D}(v)\ge \lct(X,\D;\CF_v)$. If $(X,\D,\xi_0)$ is $\IT$-equivariantly Ding-semistable for filtrations, then 
$$A_{X,\D}(v)-A(\xi_0)S(\xi_0;v)\ge \BD(\CF_v) \ge 0 $$
for any $v\in \Val^{\IT,*}_{X,x}$. Hence $\delta_\IT(X,\D,\xi_0)\ge 1$. 

Conversely, for any $\IT$-invariant linearly bounded filtration $\CF$, let $v$ be a minimizer of $\lct(X,\D;\CF)$ such that $v(\CF)=1$. Then $\CF\seq \CF_v$. Hence 
$$A_{X,\D}(v) = \lct(X,\D;\CF)\le \lct(X,\D;\CF_v)\le A_{X,\D}(v),\quad 
S(\xi_0;\CF) \le S(\xi_0;\CF_v). $$
Hence $\BD(\CF) \ge \BD(\CF_v) = A_{X,\D}(v) - A(\xi_0)S(\xi_0;v) \ge 0$. 
\end{proof}

\subsection{Delta minimizers are Koll\'ar valuations}

We have the following analog of \cite[Theorem 1.2]{LXZ22}, which can be viewed as a local version of optimal destabilization. This was given by \cite{Hua22}. We will give a proof based on Xu-Zhuang's resolution of the stable degeneration conjecture \cite{XZ22}. Note that for any log Fano cone $(X,\D,\xi_0)$, we always have $\delta_\IT(X,\D,\xi_0)\le 1$ by Lemma \ref{Lemma: delta le 1}. 
\begin{thm}
\label{Theorem. Local optimal destablization}
Let $(X,\D,\xi_0)$ be a log Fano cone. If $\delta_{\IT}(X,\D,\xi_0)$ admits a quasi-monomial minimizer $v_0$, then $v_0$ is a Koll\'ar valuation, and there exists a Koll\'ar component minimizing $\delta_{\IT}(X,\D,\xi_0)$. 
\end{thm}


\begin{rmk} \rm \label{Remark. non-existence of delta minimizer}   If $\delta_{\IT}(X,\D,\xi_0)<1$, then any $v\in\Val_{X,x}^{\IT,*}$ would not minimize $\delta_{\IT}(X,\D,\xi_0)$. Since for $0<a\ll 1$, we still have $v_{-a\xi_0}\in \Val_{X,x}^{\IT,*}$. Hence \begin{eqnarray*} \frac{A_{X,\D}(v_{-a\xi_0})}{A(\xi_0)S(\xi_0;v_{-a\xi_0})} = \frac{A_{X,\D}(v) - aA(\xi_0)}{A(\xi_0)S(\xi_0;v) - aA(\xi_0)} < \frac{A_{X,\D}(v)}{A(\xi_0)S(\xi_0;v)} < 1. \end{eqnarray*}  
If $\delta_{\IT}(X,\D,\xi_0)=1$, then it is minimized by $\wt_{\xi}$ for any $\xi\in\Reeb$. We will see in the proof of Theorem \ref{Theorem: Intro. Ding-ps for special test configurations implies Ding-ps} that $\delta_\IT(X,\D,\xi_0)$ admits non-product minimizer in $\Val_{X,x}^{\IT,*}$ under the assumption of strict Ding-semistability for normal test configurations.  \end{rmk}

Let us recall the basic notions in the higher-rank finite generation theory developed by \cite{LXZ22,XZ22} (see also \cite{Che25}). 
A $\IQ$-complement $\Gamma$ of $(X,\D)$ is called {\it special} with respect to a log smooth (resp. toroidal) model $\pi: (Y,E)\to (X,\D)$ if $\pi_*^{-1}\Gamma\ge G$ for some effective ample $\IQ$-divisor $G$ on $Y$ whose support does not contain any stratum of $(Y,E)$. 
A model $\pi: (Y,E)\to (X,\D)$ is said to be of {\it qdlt Fano type} if there exists an effective $\IQ$-divisor $D$ on $Y$ such that $D+E\ge \pi_*^{-1}\D$, $\lfloor D+E \rfloor = E$, $(Y,D+E)$ is qdlt (that is, quotient of dlt singularities by abelian groups) and $-(K_Y+D+E)$ is ample. In this case, the quasi-monomial complex $\QM(Y,E)\seq \Val_{X,x}^*$ is well-defined (\cite[Definition 2.8]{XZ22}) and is a simplicial cone (\cite[Lemma 2.3]{XZ22}), which is called a {\it special} quasi-monomial cone. Any valuation $v$ in $\QM(Y,E)$ is called a {\it Koll\'ar valuation} (or special valuation). Moreover, $v$ is called a {\it Koll\'ar component} if it is divisorial. We have the following characterization of Koll\'ar valuations. 

\begin{thm}\cite[Theorem 4.1]{XZ22}
\label{Theorem. XZ22 higher rank finite generation}
Let $x\in(X=\Spec(R),\D)$ be a klt singularity, and let $v\in\Val_{X,x}$ be a
quasi-monomial valuation. Then the following are equivalent. 
\begin{enumerate}
\item The associated graded ring $\Gr_vR$ is finitely generated and the central fiber $(X_v,\D_v)$
of the induced degeneration is klt.
\item There exists a special $\IQ$-complement $\Gamma$ of $(X,\D)$ with respect to some toroidal model $(Y',E')$ such that $v\in \LC(X,\D+\Gamma)\cap \QM(Y',E')$.
\item There exists a qdlt Fano type model $\pi: (Y,E)\to (X,\D)$ such that $v\in \QM(Y,E)$. 
\end{enumerate}
\end{thm}

\begin{thm}\cite[Corollary 4.10]{XZ22}
\label{Theorem. Special cone: geodesic = line sigment}
Let $x\in(X=\Spec(R),\D)$ be a klt singularity, and $\pi: (Y,E=E_1+\cdots+E_r)\to (X,\D)$ be a qdlt Fano type model. Let $\CF_i = \CF_{\ord_{E_i}}$. Then for any $\alpha=(\alpha_1,\cdots,\alpha_r) \in\IR^r_{\ge 0}$, we have $\CF_\alpha = \CF_{v_\alpha}$ (see Lemma \ref{Lemma. geodesic seq line-sigment}). 
\end{thm}

\begin{cor}
\label{Corollary. S is linear on special cone}
Let $(X=\Spec(R),\D,\xi_0)$ be a log Fano cone singularity, and $\pi: (Y,E=E_1+\cdots+E_r)\to (X,\D)$ be a qdlt Fano type model. Then the function $\QM(Y,E)\to \IR_{\ge 0},\,\, v\mapsto S(\xi_0; v)$ is linear, and the function $\QM(Y,E)\to \IR_{\ge 0},\,\, v\mapsto \lam_\max(\xi_0; v)$ is convex. 
\end{cor}


A valuation $v$ over $(X,\D)$ is called {\it weakly special} if it is a log canonical place of some $\IQ$-complement of $(X,\D)$. We have the following characterization of weakly special valuations. 

\begin{thm}
\label{Theorem. local weakly special criterion}
A valuation $v$ over $(X,\D)$ is weakly special if and only if $v$ is quasi-monomial and $\lct(X,\D;\CF_v) = A_{X,\D}(v)$. 
\end{thm}

\begin{proof}
For the `only if' part, we have naturally $v(\CF_v)=1$ and $\lct(X,\D;\CF_v) \le A_{X,\D}(v)$. Let $\Gamma$ be a $\IQ$-complement of $(X,\D)$ such that $v\in\LC(X,\D+\Gamma)$, then $v(\Gamma)=A_{X,\D}(v)$. Choose $k\in\IN$ such that $k\Gamma$ is Cartier, then $k\Gamma = \div(s)$ for some $s \in \CF_{v}^{kA_{X,\D}(v)}$. Hence $\lct(X,\D;\CF_v^{kA_{X,\D}(v)}) \ge \lct(X,\D;k\Gamma) \ge \frac{1}{k}$, so $\lct(X,\D;\CF_v) \ge A_{X,\D}(v)$. 

For the `if' part, it was already proved by \cite[Proof of Lemma 3.2]{XZ22} (which only used the condition $\lct(X,\D+D;\fa_\bu(v_0)) = A_{X,\D+D}(v_0)$ on $v_0$, instead of $v_0$ being the normalized volume minimizer). 
For the sake of completeness, we recall the argument. 

By \cite{HMX14}, there exists $\vep>0$ depending only on $\dim X$ and coefficients of $\D$ such that, for any birational morphism $\pi:Y\dashrightarrow X$ and any reduced divisor $E$ on $Y$, the pair $(Y,\pi_*^{-1}\D+(1-\vep)E)$ is log canonical if and only if $(Y,\pi_*^{-1}\D+E)$ is.

We denote by $\fa_\bu =\fa_\bu(v)$ and $c=A_{X,\D}(v)$. 
Since $v\in \Val_{X,o}$ is a quasi-monomial valuation, there exists a quasi-monomial simplicial cone $\sigma\seq \Val_{X,o}$ containing $v$. We may assume that $\sigma$ has minimal dimension. The functions $w\mapsto A_{X,\D}(w)$ and $w\mapsto w(\fa_\bu^c)$ are linear and concave on $\sigma$ respectively. Hence the function $A_{X,\D+\fa_\bu^c}(-): \sigma \to \IR$, 
\begin{eqnarray}
\label{Eqnarray: function A(-) on sigma}
w\mapsto A_{X,\D+\fa_\bu^c}(w) = A_{X,\D}(w) - w(\fa_\bu^c)
\end{eqnarray}
is convex on $\sigma$. In particular, it is locally Lipschitz on $\sigma$. Let $|\cdot|$ be the Euclidean norm on $\sigma$. There exists an open neighbourhood $U\seq \sigma$ of $v$ and $C=C(U)>0$ such that 
\begin{eqnarray*}
|A_{X,\D+\fa_\bu^c}(w) -
   A_{X,\D+\fa_\bu^c}(v) |
\le C|w-v|, \quad \forall w\in U. 
\end{eqnarray*}
On the other hand, $A_{X,\D+\fa_\bu^c}(w)\ge 0$ for any $w\in U$ since $v$ compute $\lct(X,\D; \fa_\bu^c)=1$. Hence
\begin{eqnarray}
0 \le 
A_{X,\D+\fa_\bu^c}(w) = 
|A_{X,\D+\fa_\bu^c}(w) -
   A_{X,\D+\fa_\bu^c}(v) |
\le C|w-v|. 
\end{eqnarray}
By Diophantine approximation \cite[Lemma 2.7]{LX18}, there exist divisorial valuations $v_1,\cdots, v_r$ and positive integers $q_1,\cdots,q_r,c_1,\cdots,c_r$ such that 
\begin{itemize}
\item $\{v_1,\cdots,v_r\}\seq U$ spans a quasi-monomial simplicial cone in $\sigma$ containing $v$; 
\item for any $1\le i\le r$, there exists a prime divisor $E_i$ over $X$ such that $q_iv_i=c_i\ord_{E_i}$; 
\item $|v_i-v|< \frac{\vep}{2Cq_i}$ for any $1\le i\le r$. 
\end{itemize}
In particular, 
\begin{eqnarray}
A_{X,\D+\fa_\bu^c}(E_i) 
= \frac{q_i}{c_i}\cdot A_{X,\D+\fa_\bu^c}(v_i) 
\le \frac{q_i}{c_i}\cdot C|v_i-v|
< \frac{q_i}{c_i}\cdot C \cdot \frac{\vep}{2 C q_i} 
\le \frac{\vep}{2}. 
\end{eqnarray}
By assumption $(X,\D+\fa_\bu^c)$ is log canonical, hence we may choose $0<\vep'\ll 1$ and $m\gg 0$ such that $(X,\D+\frac{c-\vep'}{m}\fa_m)$ is klt and   
\begin{eqnarray*}
a_i = A_{X,\D+\frac{c-\vep'}{m}\fa_m} (E_i) 
&=& c\cdot (\ord_{E_i}(\fa_\bu)-\frac{1}{m}\ord_{E_i}(\fa_m)) \\
&+& \frac{\vep'}{m}\ord_{E_i}(\fa_m)
+ A_{X,\D+\fa_\bu^c}(E_i) 
\,\,\,<\,\,\, \vep. 
\end{eqnarray*}

By \cite[Corollary 1.4.3]{BCHM10}, there exists a $\IQ$-factorial birational model $\pi:Y\to X$ extracts precisely $E_1,\cdots, E_r$. Choose general $s\in\fa_m$ and let $D_m=\frac{1}{m} \div(s)$. Then 
\begin{eqnarray} 
\label{Eqnarray: crepant pullback 1}
K_Y+\pi_*^{-1}(\D+(c-\vep')D_m)+\sum_{i=1}^r (1-a_i)E_i   \sim_{\IQ} \pi^*(K_X+\D+(c-\vep')D_m). 
\end{eqnarray}
In particular, $\pi^*(K_X+\D+(c-\vep')D_m)\ge K_Y+\pi_*^{-1}\D+(1-\vep)E$. Since $(X,\D,(c-\vep')D_m)$ is klt, the pair $(Y,\pi_*^{-1}\D+(1-\vep)E)$ is log canonical. Hence $(Y,\pi_*^{-1}\D+E)$ is also log canonical by our choice of $\vep$. 

Note that $Y$ is of Fano type over $X$ by (\ref{Eqnarray: crepant pullback 1}). We may run a $-(K_Y+\pi_*^{-1}\D+E)$-MMP over $X$ and get a good minimal model $\phi: Y\dashrightarrow \oY$ such that $-(K_\oY+\opi_*^{-1}D+\oE)$ is nef over $X$, where $\opi: \oY\to X$ is the induced birational morphism and $\oE = \phi_*E$. 
Since $(X,\D+(c-\vep')D_m)$ is log canonical, so is $(\oY,\opi_*^{-1}\D+(1-\vep)\oE)$. By our choice of $\vep$, we see that $(\oY,\opi_*^{-1}\D+\oE)$ is also log canonical. The MMP process shows that 
\begin{eqnarray}
\label{Eqnarray. (-K)-MMP A_X(v)-decending}
A_{Y,\pi_*^{-1}\D+E}(F)
\ge
A_{\oY,\opi_*^{-1}\D+\oE}(F) \ge 0, 
\end{eqnarray}
for any prime divisor $F$ over $Y$, and the first equality holds if and only if $\phi$ is an isomorphism at the generic point of $C_Y(F)$. 
Hence $\phi$ is an isomorphism at the generic point of each log canonical center of $(Y,\pi_*^{-1}\D+E)$. In particular, $\phi$ is an isomorphism at any stratum of $E$. Hence
\begin{eqnarray}
\label{Eqnarray. crepant pullback of (-K)-MMP}
\phi^*(K_\oY+\opi_*^{-1}\D+\oE) - (K_Y+\pi_*^{-1}\D+E) = \sum_i(1-A_{Y,\opi_*^{-1}\D+\oE}(F_i))(F_i) \ge 0,  
\end{eqnarray}
where $F_i$ are prime divisors extracted by $\phi$.

Since $Y$ is of Fano type over $X$ and $\phi: Y\dashrightarrow \oY$ is a birational contraction, we have that $\oY$ is also of Fano type over $X$. On the other hand, $-(K_{\oY}+\opi_*^{-1}\D+\oE)$ is nef over $X$. Hence $-(K_{\oY}+\opi_*^{-1}\D+\oE)$ is semiample over $X$ and $(\oY,\opi_*^{-1}\D+\oE)$ admits a $\IQ$-complement by Bertini theorem. By (\ref{Eqnarray. crepant pullback of (-K)-MMP}), $(Y,\pi_*^{-1}\D+E)$ also admits a $\IQ$-complement $\Theta$. Then $\Gamma=\pi_*\Theta$ is a $\IQ$-complement of $(X,\D)$ such that $E_i\in \LC(X,\D+\Gamma)$. In particular, $v\in \LC(X,\D+\Gamma)$. 
\end{proof}

The following theorem gives the specialty of delta minimizers. 

\begin{thm}
\label{Theorem. delta minimizer is very special}
Let $(X,\D,\xi_0)$ be a log Fano cone singularity of dimension $n$. If $v_0\in\Val_{X,x}^{\IT,*}$ minimizes $\delta=\delta_\IT(X,\D,\xi_0)$, then for any $\IT$-invariant effective $\IQ$-Cartier $\IQ$-divisor $D$ on $X$ such that $\wt_{\xi_0}(D)<\frac{\delta A(\xi_0)}{n}$, we have  
\begin{eqnarray*}
\lct(X,\D+D;\CF_{v_0}) = A_{X,\D+D}(v_0). 
\end{eqnarray*} 
In other words, $v_0$ minimizes $\lct(X,\D+D;\CF_{v_0})$. 
\end{thm}

\begin{proof}
Since $v_0(\CF_{v_0})=1$, we naturally have `$\le$'. Denote by $c=A_{X,\D+D}(v_0)$. It suffices to show that $\frac{A_{X,\D+D}(v)}{v(\CF_{v_0})} \ge c$, that is, $A_{X,\D}(v)\ge v(D)+ c\cdot v(\CF_{v_0})$ for any $v\in\Val_{X,x}^{\IT,*}$. 

Rescale $\xi_0$ such that $A(\xi_0)=1$. By definition of $\delta=\delta_\IT(X,\D,\xi_0)$, we have  
\begin{eqnarray}
\frac{A_{X,\D}(v)}{S(\xi_0;v)}
\ge \frac{A_{X,\D}(v_0)}{S(\xi_0;v_0)} 
= \delta.  
\end{eqnarray} 
Let $\Gamma$ be a $\IT$-invariant $m$-basis type divisor of $\xi_0$ compatible with both $\CF_{v_0}$ and $\CF_D$. Then we have 
\begin{eqnarray}
v_0(\Gamma) &\ge& S_m(\xi_0;v_0), \\
\Gamma &\ge& S_m(\xi_0;\CF_D)\cdot D. 
\end{eqnarray} 
Note that $S_m(\xi_0;v_0) \to S(\xi_0;v_0) = \delta^{-1} A_{X,\D}(v_0)$ 
as $m\to \infty$. Choose a sequence of positive numbers $\delta_m$ such that $\delta_m\to \delta$ and $\delta_m S_m(\xi_0;v_0) \ge \delta S(\xi_0;v_0)$. We have $\delta_m \Gamma \ge D$ for $m\gg0$ since 
\begin{eqnarray*}
\delta_m S_m(\xi_0;\CF_D) \to \delta S(\xi_0;\CF_D) = \frac{\delta}{n\cdot \wt_{\xi_0}(D)} >1, 
\end{eqnarray*} 
where the equality follows from Lemma \ref{Lemma. S-invariant of effective divisor}. 
By our choice of $\delta_m$, we have 
\begin{eqnarray*}
v_0(\delta_m\Gamma - D) \ge \delta S(\xi_0;v_0) -v_0(D) = A_{X,\D}(v_0) -v_0(D) = c. 
\end{eqnarray*} 
Hence we see that $v(\delta_m\Gamma - D) \ge c\cdot v(\CF_{v_0})$. 

On the other hand, note that $v(\Gamma) \le S_m(\xi_0;v)$. Fix $\vep>0$. For $m\gg0$, we have $v(\delta_m\Gamma) \le (1+\vep)\delta S(\xi_0;v) \le (1+\vep)A_{X,\D}(v)$. Hence 
\begin{eqnarray*}
(1+\vep)A_{X,\D}(v) \ge v(\delta_m\Gamma) \ge v(D) + c\cdot v(\CF_{v_0}). 
\end{eqnarray*} 
We conclude by letting $\vep\to 0$. 
\end{proof}

\begin{proof}[Proof of Theorem \ref{Theorem. Local optimal destablization}]
We first show that the $\IT$-invariant quasi-monomial minimizer $v_0$ of $\delta_\IT(X,\D,\xi_0)$ is a Koll\'ar valuation. 
Let $\pi': (Y',E')\to (X,\D)$ be a $\IT$-equivariant log resolution whose exceptional locus supports a $\IT$-invariant effective $\pi$-antiample $\IQ$-divisor $F$, and $v_0\in \QM(Y',E')$. Choose $0<\vep\ll1$ such that $L=-\pi^*(K_X+\D)-\vep F$ is ample and $\wt_{\xi_0}(\vep F) < \frac{\delta A(\xi_0)}{n}$. Hence there exists a $\IT$-invariant effective $\IQ$-divisor $G\sim_\IQ L$ such that $v_0(G)=0$. Let $D=\pi_*G\sim_{\IQ} -(K_X+\D)$. Then $D$ is $\IT$-invariant, $\pi_*^{-1}D \ge G$ and $\wt_{\xi_0}(D)=\wt_{\xi_0}(\vep F) < \frac{\delta A(\xi_0)}{n}$. By Theorem \ref{Theorem. delta minimizer is very special} and Theorem \ref{Theorem. local weakly special criterion}, there exists a $\IQ$-complement $\Gamma\ge D$ of $(X,\D)$ such that $v_0$ is a log canonical place of $(X,\D+\Gamma)$. Hence $\Gamma$ is a special $\IQ$-complement with respect to the log smooth model $\pi':(Y',E')\to (X,\D)$, and $v_0\in \LC(X,\D+\Gamma)$ is a Koll\'ar valuation by Theorem \ref{Theorem. XZ22 higher rank finite generation} $(2)\Rightarrow (3)$. 

Next we find a Koll\'ar component minimizing $\delta_\IT(X,\D,\xi_0)$. 
Since $v_0$ is a $\IT$-invariant Koll\'ar valuation, there exists a $\IT$-invariant special quasi-monomial cone $\sigma$ containing $v_0$. Note that 
\begin{eqnarray*}
\delta=\frac{A_{X,\D}(v_0)}{S(\xi_0; v_0)} \le \frac{A_{X,\D}(v)}{S(\xi_0; v)}, \quad \forall v\in\sigma .
\end{eqnarray*}
The equality holds since both $A_{X,\D}(-)$ and $S(\xi_0;-)$ are linear on $\sigma$ by \cite[Proposition 5.1 (ii)]{JM12} and Corollary \ref{Corollary. S is linear on special cone}. We conclude by choosing $v\in\sigma_\IQ$, which is a Koll\'ar component. 
\end{proof}

Now we are ready to prove the first main theorem of the paper. 

\begin{proof}[Proof of Theorem \ref{Theorem: Intro. Ding-ps for special test configurations implies Ding-ps}]
Assume that $(X,\D,\xi_0)$ is $\IT$-equivariantly Ding-semistable. Let $(\CX,\D_\CX,\xi_0;\eta)$ be a $\IT$-equivariant normal test configuration of $(X,\D,\xi_0)$ and $\CF$ be the corresponding $\IT$-invariant $\IN$-filtration. Then $\CF$ is saturated by \cite[Definition 2.20]{LWX18} and \cite[Lemma 3.19]{BLQ22}. There exists a $\IT$-invariant quasi-monomial valuation $v$ minimizing $\lct(\CF)$ by \cite{Xu19}. Rescale $v$ such that $v(\CF)=1$. Then $\CF\seq \CF_v$ and 
\begin{eqnarray}
\label{Eqnarray. lct(a_bu) = lct(a_bu(v)) = A(v)}
A_{X,\D}(v)=\lct(\CF)\le\lct(\CF_v)\le A_{X,\D}(v).
\end{eqnarray}
Hence $\lct(\CF)=A_{X,\D}(v)$ and $S(\xi_0;\CF) \le S(\xi_0;v)$. We see that $\BD(\CF)\ge \BD(v)\ge 0$ since $(X,\D,\xi_0)$ is $\IT$-equivariantly Ding-semistable. 

Assume that $(X,\D,\xi_0)$ is not $\IT$-equivariantly Ding-polystable for normal test configurations, then we may assume that $(\CX,\D_\CX,\xi_0;\eta)$ is not of product-type and has vanishing $\BD$. Hence $\CF$ is not of product type. We have 
$$0= \BD(\CX,\D_\CX,\xi_0;\eta) \ge \BD(\CF)\ge \BD(v)\ge 0. $$
Hence $\BD(\CF)=\BD(v)= 0$, and $S(\xi_0;\CF) = S(\xi_0;v)$. Since $\CF\seq \CF_v$, we see that $\CF = \CF_v$ by Lemma \ref{Lemma. saturated fa seq fb, S(fa)=S(fb) implies fa=fb}, hence $v$ is not of product type. Then $v$ minimizes $\delta_\IT(X,\D,\xi_0)=1$, and it is a Koll\'ar valuation by Theorem \ref{Theorem. Local optimal destablization}. We can perturb $v$ and get a $\IT$-invariant Koll\'ar component $E$ minimizing $\delta_\IT(X,\D,\xi_0)=1$, which is not of product type since this holds for $v$. It induces a non-product special test configuration of $(X,\D,\xi_0)$ with vanishing $\BD$. We conclude that $(X,\D,\xi_0)$ is not Ding-polystable for special test configurations. 
\end{proof}

\section{Relations between K/Ding-stability}
\label{Section. Relations between K/Ding-stability}

In order to show that all the four polystability conditions in Definition \ref{Definition. K/Ding polystability} are equivalent, we still need to establish the inequality $\Fut\ge\BD$ for any normal test configurations similar to \cite{LW24}. 
We compare the Fut/Ding invariants of quasi-regular log Fano cones with the Fut/Ding invariants of the corresponding base Fano varieties in this section. We will prove the following theorem.  

\begin{thm}
Let $x\in(X,\D)$ be a klt singularity with a good $\IT$-action. For any $\xi_0\in\Bt_\IR^+$ and any test configuration $(\CX,\D_\CX,\xi_0;\eta)$ of $(X,\D,\xi_0)$, we have 
$\Fut(\CX,\D_\CX,\xi_0;\eta) \ge \BD(\CX,\D_\CX,\xi_0;\eta). $
\end{thm}

\subsection{Log Fano varieties}

We first recall the basic notions in the algebraic K-stability theory of log Fano varieties, see for example \cite{Xu24}. 
Let $(V,\D_V)$ be a log Fano variety, that is, $(V,\D_V)$ is a projective klt pair with ample anti-canonical divisor $L = -(K_V+\D_V)$. Choose a positive integer $l_0$ such that $l_0(K_V+\D_V)$ is Cartier. We denote by $R(L) = \oplus_{m\in l_0\IN} R_m(L)$ the anti-canonical ring of $(V,\D_V)$, where $R_m(L)=H^0(V,mL)$.

\begin{defi}\rm
\label{Definition. filtration of Fano varieties}
A \textit{filtration} $\CF$ on $R(L)$ is a collection of subspaces $\{\CF^\lam R_m\seq R_m\}_{\lam\in\IR}$ such that:
\begin{enumerate}
\item $\CF^\lam R_m \supseteq \CF^{\lam'}R_m $ when $\lam \le \lam'$, 
\item  $\CF^\lam R_m=\CF^{\lam-\epsilon}R_m$ when $0<\epsilon \ll 1$, 
\item $\CF^\lam R_m \cdot \CF^{\lam'}R_{m'} \subseteq \CF^{\lam+\lam'}R_{m+m'}$, and
\item $\CF^\lam R_m = R_m$ for $\lam \ll 0$ and $\CF^\lam R_m = 0$ for $\lam \gg 0$. 
\end{enumerate}
It is called {\it linearly bounded} if moreover there is a constant $C>0$ such that $\CF^{-mC}R_m =R_m$ and $\CF^{mC}R_m=0$ for any $m\in l_0\IN$. 
\end{defi}

We define 
\begin{eqnarray*} 
\lam^{(m)}_\max(\CF) 
&:=& \max\{\lam\in \IR\mid \CF^{\lam}R_m \ne 0\}, \\
S_m(\CF) 
&:=&
\sum_{\lam}\frac{\lam}{m} \cdot 
\frac{\dim \gr_\CF^{\lam} R_m}{\dim R_m}. \qquad
\end{eqnarray*} 
By \cite{BJ20}, the limits exist as $m\to \infty$, 
\begin{eqnarray*} 
\lam_\max(\CF) 
&:=& \mathop{\sup}_{m\in \IN} \frac{\lam^{(m)}_\max}{m} 
= \mathop{\lim}_{m\rightarrow \infty} \frac{\lam^{(m)}_\max}{m} , \\
S_{V,\D_V}(\CF) &:=& \mathop{\lim}_{m\to \infty} S_m(\CF). 
\end{eqnarray*}

For any filtration $\CF$ on $R(L)$ and $a\in\IR_{>0}, b\in \IR$, we define the $a$-{\it rescaling} and $b$-{\it shift} of $\CF$ by
$$(a\CF)^\lam R_m := \CF^{\lam/a}R_m, \quad
\CF(b)^\lam R_m := \CF^{\lam-bm} R_m. $$

\begin{defi}[log canonical slopes]\rm 
\label{Definition: log canonical slope}
Let $L=-(K_V+\D_V)$ and $\CF$ be a linearly bounded filtration on $R(L)$. The {\it base ideal sequence} $I^{(t)}_\bu = \{I_{m,mt}\}_{m\in l_0\IN}$ of $\CF$ is defined by 
\begin{eqnarray*}
I_{m,mt}
\,\,\,=\,\,\,
I_{m,mt}(\CF)
\,\,\,:=\,\,\,{\rm im}
\Big(\CF^{mt}H^0(X,mL)\otimes \CO(-mL)\to \CO\Big)
\end{eqnarray*}
for any $m\in l_0\IN$ and $t\in\IR$. The {\it log canonical slope} of $\CF$ is defined by
\begin{eqnarray*}
\mu(\CF)
\,\,\,=\,\,\, 
\mu_{X,\D}(\CF)
\,\,\,:=\,\,\,
\sup\Big\{
t\mid \lct(X,\D;I^{(t)}_\bu)\ge 1
\Big\}. 
\end{eqnarray*}
Note that $I^{(t)}_\bu=0$ (hence $\lct(X,\D;I^{(t)}_\bu)=0$) when $t>\lam_\max(\CF)$. We have $\mu(\CF)\le \lam_\max(\CF)$. 
\end{defi}

\begin{lem}
For any $a\in\IR_{>0}, b\in \IR$, we have $\mu(a\CF)=a\mu(\CF)$ and $\mu(\CF(b)) =\mu(\CF)+b$. 
\end{lem}

A valuation $v$ over $(V,\D_V)$ is called {\it weakly special} if it is a log canonical place of some $\IQ$-complement of $(V,\D_V)$. 
We have the following characterization of weakly special valuations. 
  
\begin{thm}\cite[Theorem 2.13]{Wang24b}
\label{Theorem. global weakly special criterion}
A valuation $v$ over $(V,\D_V)$ is weakly special if and only if $v$ is quasi-monomial and $\mu_{V,\D_V}(\CF_v) = A_{V,\D_V}(v)$. 
\end{thm}

\begin{defi}\rm 
For any linearly bounded filtration $\CF$, the {\it Ding invariant} is defined by 
\begin{eqnarray*}
\BD_{V,\D_V}(\CF) = \mu_{V,\D_V}(\CF) - S_{V,\D_V}(\CF). 
\end{eqnarray*}
\end{defi}

\subsection{Quasi-regular log Fano cone singularities}
Let $(X,\D,\xi_0)$ be a log Fano cone singularity, where $X=\Spec(R)$ and $R=\oplus_{\alpha\in M}R_\alpha$. 
If $\xi_0$ is rational ($\xi_0\in N_\IQ$), then $(X,\D,\xi_0)$ is said to be {\it quasi-regular}. Let 
$$R^{\xi_0} := \bigoplus_{m\in\IN} R^{\xi_0}_m, \quad
R^{\xi_0}_m:=\bigoplus_{\la \alpha, \xi_0\ra =m} R_\alpha. $$ 
Then $R^{\xi_0} \seq R$ is a $\IT$-invariant subring, and the equality holds if and only if $\xi_0\in N$. In the following, we assume that $\xi_0\in N$.  
In this case, $\xi_0$ generates a $\IG_m$-action on $X$, which we denote by $\la\xi_0\ra$. 
Then $V=\Proj (R^{\xi_0})$ is the quotient of $X\setminus\{x\}$ by the $\la\xi_0\ra$-action, and the natural map $\tau: X\setminus\{x\}\to (V,\D_V^0)$ is a Seifert $\IG_m$-bundle (see \cite{Kol13}), where $\D_V^0=\sum_j(1-\frac{1}{b_j})D_j^0$ is the branched divisor of $\tau$. We have $K_{X\setminus\{x\}} = \tau^*(K_V+\D_V^0)$. Write $\D=\sum_ia_i\D_i$, where $\D_i$ are $\la\xi_0\ra$-invariant. Hence we have $\D_i = \frac{1}{m_i} \tau^*D_i^1$ for some prime divisor $D_i^1$ on $V$ and $m\in \IZ_{\ge 1}$. If we denote by $\D_V^1 := \sum_i \frac{a_i}{m_i}D_i^1$, then $\D|_{X\setminus\{x\}} = \tau^*\D_V^1$. Finally, let $\D_V = \D_V^0+\D_V^1$. We have $(K_X+\D)|_{X\setminus\{x\}} = \tau^*(K_V+\D_V)$. By \cite[Section 3.1]{Kol13}, we see that $(V,\D_V)$ is a log Fano pair admitting a $\IT_V=(\IT/\la\xi_0\ra)$-action. Since $V=\Proj(R^{\xi_0})$, there exists an ample $\IQ$-divisor $L_0$ on $V$ such that 
\begin{eqnarray}
\label{Eqnarray: R_m = H^0(V,mL)}
R^{\xi_0}_m\cong H^0(V,mL_0)
\end{eqnarray}
for sufficiently divisible $m\in\IN$. We have 
\begin{eqnarray}
\label{Eqnarray: -(K+D) = A(xi_0)L_0}
-(K_V+\D_V) \sim_\IQ A(\xi_0)L_0. 
\end{eqnarray}
Indeed, let $p: \tX=\Spec(\oplus_{m\in l_0\IN} \CO_V(mL_0)) \to V$ be the total space of the line bundle $-l_0L_0$. Then the natural morphism $\pi:\tX\to X$ induced by $H^0(V,mL_0)$ is birational which is an isomorphism over $X\setminus \{x\}$ and has exceptional locus $\Ex(\pi) =: V_0 \cong V$. We see that $p|_{X\setminus\{x\}} =\tau$. Hence $$K_V+\D_V = (K_{\tX}+\pi_*^{-1}\D+V_0)|_{V_0} \sim_\IQ A_{X,\D}(V_0)V_0|_{V_0} = -A_{X,\D}(V_0)l_0L_0. $$
On the other hand, note that $\wt_{\xi_0}=l_0\ord_{V_0}$. We get the $\IQ$-linear equivalence (\ref{Eqnarray: -(K+D) = A(xi_0)L_0}). 

We see that $(V,\D_V)$ does not depend on the rescaling of $\xi_0$, but $L_0$ does. 

By Definition \ref{Definition. filtrations of klt singularities} and Definition \ref{Definition. filtration of Fano varieties}, we have the following correspondence. 

\begin{lem}
\label{Lemma. filtration correspondence of quasi-regular cone}
Any $\IT$-invariant linearly bounded filtration $\CF$ on $R$ restricts to a $\IT$-invariant linearly bounded filtration with $\lam_\min(L_0;\CF)>0$ on $R(L_0)\cong R^{\xi_0} \seq R$. Conversely, any $\IT$-invariant linearly bounded filtration $\CF$ on $R(L_0)$ with $\lam_\min(L_0;\CF)>0$ generates a $\IT$-invariant linearly bounded filtration on $R$. Moreover, 
\begin{eqnarray*}
\lam_\min(\xi_0;\CF) &=& \lam_\min(L_0;\CF), \\
\lam_\max(\xi_0;\CF) &=& \lam_\max(L_0;\CF), \\
S(\xi_0;\CF) &=& S(L_0;\CF), \\
\CF_{b\xi_0} &=& \CF(b)R(L_0).  
\end{eqnarray*}
\end{lem}
\begin{proof}
The first three equalities follow from (\ref{Eqnarray: R_m = H^0(V,mL)}). We only need to prove the last one. For any $\alpha\in M$ such that $\la\alpha,\xi_0\ra=m$, we have 
\begin{eqnarray*}
\CF_{b\xi_0}^\lam R_\alpha 
= \CF^{\lam - b\la\alpha,\xi_0\ra} R_\alpha 
= \CF^{\lam - bm} R_\alpha
= \CF(b)^\lam R_\alpha. 
\end{eqnarray*}
We conclude by taking direct sum for all $\alpha \in M$. 
\end{proof}

Since $L=-(K_V+\D_V)\sim_\IQ A(\xi_0)L_0$, we have 
\begin{eqnarray}
\label{Eqnarray. S=S}
A(\xi_0)S(\xi_0; \CF) &=& S_{V,\D_V}(\CF), \\
\CF_{b\xi_0} &=& \CF(A(b\xi_0))R(L)
\end{eqnarray}
for any $\IT$-invariant linearly bounded filtration $\CF$ on $R$ and any $b\ge0$. 


\begin{lem}
\label{Lemma. valuation correspondence on quasi-regular cone}
The quotient morphism $\tau: X\setminus \{x\} \to V$ induces an isomorphism of the valuations spaces: 
$\tau^*: \Val_V^{\IT_V,*} \times \IR_{>0} \cong \Val^{\IT,*}_{X,x},\, (v,b)\mapsto v_{b\xi_0}$, where $\IT_V=\IT/\la\xi_0\ra$. Moreover,  
\begin{eqnarray}
\label{Eqnarray. A(v_xi)=A(v)+A(xi)}
A_{X,\D}(v_{b\xi_0}) = A_{V,\D_V}(v) + bA(\xi_0). 
\end{eqnarray}
\end{lem}
\begin{proof}
The pull-back map and $\xi_0$-twist give us an embedding $\Val_V^{\IT_V,*} \times \IR_{>0} \seq \Val^{\IT,*}_{X,x}$, which is surjective by \cite[Lemma 4.2]{BHJ17}. The equality (\ref{Eqnarray. A(v_xi)=A(v)+A(xi)}) follows from \cite[Lemma 2.6 and 2.7]{LW24}. 
\end{proof}

Recall that a {\it $\IQ$-complement} $\Gamma_V$ of the log Fano pair $(V,\D_V)$ is an effective $\IQ$-divisor on $V$ such that $(V,\D_V+\Gamma_V)$ is log canonical and $K_V+\D_V+\Gamma_V\sim_\IQ 0$; a {\it $\IQ$-complement} $\Gamma$ of the klt singularity $x\in(X,\D)$ is an effective $\IQ$-divisor on $X$ such that $(X,\D+\Gamma)$ is log canonical and $x\in X$ is a log canonical center. 
The following Lemma follows directly from \cite[Theorem 3.1]{Kol13}. 

\begin{lem}
\label{Lemma. weakly special valuations correspondence}
Let $\Gamma_V$ be a $(\IT/\la\xi_0\ra)$-invariant $\IQ$-complement of $(V,\D_V)$, and let $\Gamma$ be the closure of $\tau^*\Gamma_V$ on $X$. Then $(X,\D+\Gamma)$ is log canonical. 
Conversely, for any $\IT$-invariant $\IQ$-complement $\Gamma$ of $(X,\D)$, there exists a $\IQ$-complement $\Gamma_V$ of $(V,\D_{V})$ such that $\Gamma|_{X\setminus\{x\}}=\tau^*\Gamma_V$. Moreover, we have an isomorphism  
$$\LC(V,\D_V+\Gamma_V)^{\IT/\la\xi_0\ra} \times \IR_{\ge 0} \cong \LC(X,\D+\Gamma)^\IT, \, (v,b)\mapsto v_{b\xi_0}, $$
where $\LC(Y,\D_Y)^G$ is the set of $G$-invariant log canonical places of a log canonical pair $(Y,\D_Y)$ admitting an algebraic group $G$-action. 
\end{lem}

The following theorem shows that our definition of Ding invariants is the same as those in the literature \cite{LWX18,LW24}, which relies on the theory of weakly special valuations. 
\begin{thm}
\label{Theorem. lct = mu}
For any $\IT$-invariant linearly bounded filtration $\CF$ on $R$, we have 
\begin{eqnarray}
\label{Eqnarray. lct=mu}
\lct(X,\D;\CF) 
= \mu_{V,\D_V}(\CF). 
\end{eqnarray}
\end{thm}
\begin{proof}
We simply denote by 
$$\lct(\CF)=\lct(X,\D;\CF), \quad \mu(\CF)=\mu_{V,\D_V}(\CF), \quad a=A(\xi_0).$$ 
If $\CF=\CF_w$ and $w=v_{b\xi_0}$ for some weakly special valuation $v\in\Val_{V}$ (with respect to $(V,\D_V)$) and $b>0$, then $w$ is also weakly speicial (with respect to $(X,\D)$) by Lemma \ref{Lemma. weakly special valuations correspondence}. We have
\begin{eqnarray}
\label{Eqnarray. lct = mu   for weakly special valuations}
\lct(\CF_w) = A_{X,\D}(w) = A_{V,\D_V}(v)+ab = \mu(\CF_v)+ ab = \mu(\CF_v(ab)). 
\end{eqnarray}
We get the equality (\ref{Eqnarray. lct = mu   for weakly special valuations}) since $\CF_{v_{b\xi_0}} = \CF_{v,b\xi_0}=\CF_v(ab)$. 

In general, for any linearly bounded filtration $\CF$ on $R$, there exists a weakly special minimizer $v$ of $\lct(V,\D_V; I_\bu^{(\mu)}(\CF))$ by \cite{Xu19} and Theorem \ref{Theorem. global weakly special criterion}, where $\mu=\mu(\CF)$. Hence $w=v_{b\xi_0}$ is weakly special on $(X,\D)$ for any $b>0$ by Lemma \ref{Lemma. weakly special valuations correspondence}. Let $f_v(t) = v(I_\bu^{(t)})$, and denote by $f_{v,-}'(t)$ the left-derivative of $f_v$. With the same argument as \cite[Theorem 3.52]{Xu24} (see also \cite[Theorem 3.8]{Wang24b}), by rescaling $v$ such that $f_{v,-}'(\mu)=1$, we have $\CF':=\CF(A_{V,\D_V}(v)-\mu)\seq \CF_v$. Hence $A_{V,\D_V}(v) = \mu(\CF') \le \mu(\CF_v)\le A_{V,\D_V}(v)$, that is, $\mu(\CF)=\mu(\CF_v)=A_{V,\D_V}(v)$. Shifting again, we get $\CF\seq \CF_v(\mu-A_{V,\D_V}(v))=: \CF''$ and $\mu(\CF)=\mu(\CF'')$. For $b\gg0$, we have 
$$\lct(\CF)+ab= \lct(\CF_{b\xi_0}) \le \lct(\CF''_{b\xi_0}) = \mu(\CF''_{b\xi_0}) =\mu(\CF'') +ab, $$
where the second equality follows from (\ref{Eqnarray. lct = mu   for weakly special valuations}). We get $\lct(\CF)\le \mu(\CF)$. 

Conversely, by \cite{Xu19} and Theorem \ref{Theorem. local weakly special criterion}, there exists a $\IT$-invariant weakly special valuation $w$ minimizing $\lct(\CF)$. Rescale $w$ such that $w(\CF)=1$, then $\lct(\CF)=A_{X,\D}(w)$ and $\CF\seq \CF_w$. Hence $A_{X,\D}(w)=\lct(\CF)\le \lct(\CF_w)\le A_{X,\D}(w)$, that is, $\lct(\CF)=\lct(\CF_w)=A_{X,\D}(w)$. By Lemma \ref{Lemma. valuation correspondence on quasi-regular cone}, $w=v_{b\xi_0}$ for some $b>0$ and $v\in\Val_{V}^\IT$. We conclude that 
$$\mu(\CF) \le \mu(\CF_w)=\lct(\CF_w) =\lct(\CF), $$
where the first equality follows from (\ref{Eqnarray. lct = mu   for weakly special valuations}).
\end{proof}

\begin{cor}
\label{Corollary. D=D}
For any $\IT$-invariant linearly bounded filtration $\CF$ on $R$, we have 
\begin{eqnarray}
\BD_{X,\D,\xi_0}(\CF) = \BD_{V,\D_V}(\CF). 
\end{eqnarray}
\end{cor}
\begin{proof}
It follows from (\ref{Eqnarray. S=S}) and (\ref{Eqnarray. lct=mu}). 
\end{proof}

Let $(\CV,\D_{\CV};\eta)$ be a normal test configuration of $(V,\D_V)$, and choose $\lam>0$ such that $\lam(K_\CV+\D_\CV)$ is Cartier. Then we get a normal test configuration $(\CX,\D_\CX,\xi_0;\eta)$ of the log Fano cone $(X,\D,\xi_0)=C(V,\D_V;-\lam(K_V+\D_V))$, where $(\CX,\D_\CX,\xi_0) = C(\CV,\D_\CV;-(K_\CV+\D_\CV))$. We have $A(\xi_0)=\lam^{-1}$.  

\begin{thm}\cite[Lemma 6.20]{Li17}, \cite[Lemma 2.30]{LWX18}
\label{Theorem. Fut=Fut}
\begin{eqnarray}
\Fut(\CX,\D_\CX,\xi_0;\eta) = \Fut(\CV,\D_{\CV};\eta). 
\end{eqnarray}
\end{thm}

\begin{thm}\cite[Proposition 3.15]{Wu22}, \cite[Theorem 5.6]{LW24}
\label{Theorem. Fut is continuous w.r.t xi_0}
Let $x\in (X= \Spec(R),\D)$ be a klt singularity with good $\IT$-action. Let $\Reeb$ be the Reeb cone of it. For any $\IT$-equivariant test configuration $(\CX,\D_\CX;\eta)$ of $(X,\D)$ and any $\IT$-invariant linearly bounded filtration $\CF$ on $R$, the following real functions on $\Reeb$ are continuous: 
\begin{eqnarray}
\xi_0 &\mapsto& \Fut(\CX,\D_{\CX},\xi_0;\eta),  S(\xi_0;\CF), \lam_\max(\xi_0;\CF). 
\end{eqnarray}
\end{thm}

\begin{cor}
\label{Corollary. Fut ge Ding}
Let $(X,\D,\xi_0)$ be a log Fano cone singularity. For any test configuration $(\CX, \D_\CX,\xi_0;\eta)$
of $(X,\D,\xi_0)$, we have 
$\Fut(\CX, \D_\CX,\xi_0;\eta) \ge \BD(\CX, \D_\CX,\xi_0;\eta)$. 

Moreover, the equality holds if $(\CX, \D_\CX,\xi_0;\eta)$ is weakly special. 
\end{cor}
\begin{proof}
When $\xi_0\in \Reeb \cap N_\IQ$, it follows from Corollary \ref{Corollary. D=D}, Theorem \ref{Theorem. Fut=Fut} and \cite[Proposition 7.32]{BHJ17}. Hence it holds for all $\xi_0\in\Reeb$ by Theorem \ref{Theorem. Fut is continuous w.r.t xi_0}.  
\end{proof}

\begin{cor}[Equivalence of norms]
\label{Corollary. Equivalence of norms}
Let $n=\dim X$. 
For any $w\in \Val_{X,x}^{\IT,*}$, let $\CF=\CF_w$. We have 
\begin{eqnarray}
\label{Eqnarray. S - lam_min bounded by lam_max - lam_min}
\frac{1}{n} (\lam_\max-\lam_\min)(\xi_0;\CF) 
 &\le\,\,\, (S-\lam_\min)(\xi_0;\CF) &\le\,\,\, 
 (1-\frac{1}{n}) (\lam_\max-\lam_\min)(\xi_0;\CF), \\
\label{Eqnarray. S - lam_min bounded by J}
 \frac{1}{n-1} \BJ(\xi_0;\CF) 
 &\le\,\,\, (S-\lam_\min)(\xi_0;\CF) &\le\,\,\, 
 (n-1) \BJ(\xi_0;\CF).
\end{eqnarray}
\end{cor}
\begin{proof}
If $\xi_0\in \Reeb \cap N_\IQ$, we simply assume that $\xi_0\in N$ by rescaling since the above invariants are all homogeneous of degree $-1$ in $\xi_0$. We use notions as Lemma \ref{Lemma. filtration correspondence of quasi-regular cone}. Then by Lemma \ref{Lemma. valuation correspondence on quasi-regular cone}, there exists $v\in\Val_{V}^{\IT_V, *}$ and $b>0$ such that $w=v_{b\xi_0}$. By \cite[Lemma 2.6]{BJ20} and \cite[Lemma 4.1]{AZ22}, we have
\begin{eqnarray}
\label{Eqnarray. Fano case: norm equivalence}
\frac{1}{n} \lam_\max(L_0;\CF_v) 
 \le S(L_0;\CF_v) \le 
 (1-\frac{1}{n})\lam_\max(L_0;\CF_v). 
\end{eqnarray}
On the other hand, since $\CF_{v_{b\xi_0}} = \CF_{v,b\xi_0} = \CF_v(b)R(L_0)$, we have  $S(\xi_0; v_{b\xi_0}) = S(L_0; v) + b$, and $\lam_\max(\xi_0; v_{b\xi_0}) = \lam_\max(L_0; v) + b$ by Lemma \ref{Lemma. filtration correspondence of quasi-regular cone}. Note that $\lam_\min(\xi_0;v_{b\xi_0}) 
=b$. We get (\ref{Eqnarray. S - lam_min bounded by lam_max - lam_min}). Considering $(\lam_\max-\lam_\min)(\xi_0;\CF) -$ (\ref{Eqnarray. S - lam_min bounded by lam_max - lam_min}), we get (\ref{Eqnarray. S - lam_min bounded by J}). 

We conclude for any $\xi_0\in\Reeb$ by applying Theorem \ref{Theorem. Fut is continuous w.r.t xi_0}. 
\end{proof}

\section{Reduced uniform stability}
\label{Section. Reduced uniform stability}

In this section, we introduce the notion of reduced uniform K/Ding-stability of a log Fano cone, and prove that they are equivalent to K/Ding-polystability studied in the previous sections.

\subsection{Reduced uniform K/Ding-stability}
Let $(X,\D,\xi_0)$ be a log Fano cone, and $\CF$ be a $\IT$-invariant linearly bounded filtration on $R$. 
The {\it reduced $J$-norm} of $\CF$ is defined by
\begin{eqnarray}
\BJ_\IT(\xi_0;\CF) = \inf_{\xi\in \Bt_\IR^+} \BJ(\xi_0;\CF_\xi). 
\end{eqnarray}

\begin{defi}[Reduced uniform K/Ding-stability]\rm
\label{Definition. Ding-stability for filtrations}
Let $(X,\D,\xi_0)$ be a log Fano cone. Assume that $\IT\seq \Aut(X,\D,\xi_0)$ is a maximal torus. Then $(X,\D,\xi_0)$ is called {\it reduced uniformly Ding-stable} for filtrations (resp. normal test configurations, special test configurations) if there exists $\eta>0$ such that 
$$\BD(\CF)\ge \eta \BJ_\IT(\xi_0;\CF)$$ 
for any $\IT$-invariant linearly bounded filtration $\CF$ (resp. any filtration $\CF$ of $\IT$-equivariant normal test configuration, any filtration $\CF$ of $\IT$-equivariant special test configuration). 

Replacing $\BD$ by $\Fut$, we get the definition of {\it reduced uniform K-stability} for normal test configurations (resp. special test configuration). 
\end{defi}

\begin{thm}
\label{Corollary. S_T bounded by J_T}
For any $w\in\Val_{X,x}^{\IT, *}$, we have
\begin{eqnarray}
 \frac{1}{n-1} \BJ_\IT(\xi_0;w) 
 \le \inf_\xi S(\xi_0;w_\xi) \le 
 (n-1) \BJ_\IT(\xi_0;w), 
\end{eqnarray}
where the infimum runs over all $\xi\in N_\IR$ such that $w_\xi \in \Val_{X,x}^{\IT,*}$ (equivalently, $\lam_\min(\xi_0;w_\xi)>0$). 
\end{thm}

\begin{proof}
The first inequality follows directly from the first inequality in (\ref{Eqnarray. S - lam_min bounded by J}). 

For the second one, it suffices to show that for any $\vep>0$, there exists $\xi\in N_\IR$ such that $0< \lam_\min(\xi_0; w_\xi)<\vep$. Indeed, by the second inequality in (\ref{Eqnarray. S - lam_min bounded by J}), we have 
\begin{eqnarray*}
S(\xi_0;w_\xi) - \lam_\min(\xi_0; w_\xi) \le (n-1) \BJ(\xi_0;w_\xi). 
\end{eqnarray*}
Hence $S(\xi_0;w_\xi) \le (n-1) \BJ(\xi_0;w_\xi) + \vep. $ Taking infimum for $\xi$ and letting $\vep\to 0$, we will get the desired inequality. 
Now we prove the assumption. Fix $\vep>0$. Then $\lam_\min(\xi_0;w_{b\xi_0}) = \lam_\min(\xi_0;w) + b$ for any $b> -\lam_\min(\xi_0;w)$ by Lemma \ref{Lemma. lam_min xi_0 twist}. We conclude by choosing $b= \frac{1}{2}\vep -\lam_\min(\xi_0;w)$. 
\end{proof}

By Lemma \ref{Lemma: S and T, twist by xi_0}, we see that the function $\xi \to \BJ(\xi_0;\CF_\xi)$ decents to $N_\IR'=\Bt_\IR^+/(\xi_0)\to \IR$. Moreover, with the same argument of \cite[Proposition 6.6]{Xu24}, we have the following. 

\begin{lem}
\label{Lemma: J is convex w.r.t twist}
Let $\CF$ be a $\IT$-invariant linearly bounded filtration on $R$. 
Then the function $N_\IR'\to \IR, \bar{\xi}\mapsto \BJ(\xi_0;\CF_\xi)$ is convex. Moreover, there exists $C > 0$ such that   
\begin{eqnarray*}
\BJ(\xi_0;\CF_{\xi}) \ge 
\BJ(\xi_0; \xi)
-S(\xi_0;\CF), \quad 
\BJ(\xi_0; \xi) \ge C\cdot |\bar{\xi}|
\end{eqnarray*}
for any $\xi\in\Reeb$, where $|\cdot|$ is a Euclidean norm on $N_\IR'$.  
In particular, $\BJ(\xi_0; \xi) = 0$ if and only if $\bar{\xi}=0$ and
$\bar{\xi} \mapsto \BJ(\xi_0;\CF_\xi)$ admits a unique minimizer on $N_\IR'$. 
\end{lem}

\begin{proof}
For the convexity, since $\xi\mapsto S(\xi_0;\CF_\xi)$ is linear by Lemma \ref{Lemma. S(F_xi)=S(F)+S(xi)}, it suffices to show that $\xi\to \lam_\max(\xi_0; \CF_\xi)$ is convex. Let 
\begin{eqnarray*}
\label{}
\Gamma_m = \Gamma_m^{\xi_0} := \{ \alpha \in M\mid R_\alpha \ne 0, m-1 \le \la\alpha,\xi_0\ra <m \}. 
\end{eqnarray*}
By (\ref{Eqnarray. lam_max^alpha}), we have 
\begin{eqnarray}
\label{Eqnarray. lam_max^m}
\frac{1}{m} \lam_\max^{(m)}(\xi_0;\CF_\xi) = \frac{1}{m} \max\{\lam_\max^{(\alpha)}(\CF) +\la\alpha,\xi\ra \mid \alpha \in \Gamma_m \}, 
\end{eqnarray}
which is convex in $\xi$ (since the maximum of a family of linear functions is convex). Hence $\lam_\max(\xi_0;\CF_\xi) = \sup_m \frac{1}{m} \lam_\max^{(m)}(\xi_0;\CF_\xi)$ is convex in $\xi$. 
Since $\lam_{\max}^{(\alpha)}(\CF) \ge 0$ for any $\alpha$, we have 
\begin{eqnarray*}
\frac{1}{m}\lam_\max^{(m)}(\xi_0;\CF_\xi) \ge \frac{1}{m} \max\{\la\alpha,\xi\ra \mid \alpha \in \Gamma_m \} = \frac{1}{m}\lam_\max^{(m)}(\xi_0;\xi). 
\end{eqnarray*}
Taking $m\to \infty$ we get the first inequality. 

For the second inequality, since $S(\xi_0;-): \Bt_\IR^+ \to \IR_{>0}$ is linear, there exists $\alpha_0\in \interior(\sigma)$ such that $S(\xi_0;\xi) = \la\alpha_0, \xi\ra$ for any $\xi \in \Bt_\IR^+$. In particular, $\la\alpha_0, \xi_0\ra =S(\xi_0; \xi_0) = 1$. We denote by 
\begin{eqnarray*}
\BP= \BP^{\xi_0} := \{\alpha\in \sigma\mid \la\alpha,\xi_0\ra = 1\}. 
\end{eqnarray*}
Then $\alpha_0\in \interior(\BP)$. Let 
\begin{eqnarray}
\label{Eqnarray. dist alpha0 to partial P}
C = \dist(\alpha_0, \partial \BP) > 0, 
\end{eqnarray} 
be the Euclidean distance from $\alpha_0$ to $\partial \BP$. 
Then $\lam_\max(\xi_0;\xi) = \max_{\alpha\in \BP} \la\alpha, \xi\ra$ and 
\begin{eqnarray*}
\BJ(\xi_0;\xi) = \lam_\max(\xi_0;\xi) - S(\xi_0;\xi) = \max_{\alpha\in \BP} \la\alpha-\alpha_0, \bar{\xi}\ra \ge C \cdot |\bar{\xi}|, 
\end{eqnarray*}
where the second equality follows since $\la\alpha-\alpha_0, \xi_0\ra=0$. 
Hence $\BJ(\xi_0;\CF_{a\xi}) \ge aC|\bar{\xi}| - S(\xi_0;\CF) \to +\infty$ as $a\to +\infty$. In particular, $\bar{\xi}\mapsto \BJ(\xi_0;\CF_\xi)$ admits a unique minimizer on $N_\IR'$. 
\end{proof}

\begin{cor} \label{Corollary: J(CX)=0}  If $\BJ_\IT(\CF_{(\CX,\D_\CX,\xi_0;\eta)})=0$, then $(\CX,\D_\CX,\xi_0;\eta)$ is a product test configuration.  \end{cor}

\begin{proof}
Let $\CF=\CF_{(\CX,\D_\CX,\xi_0;\eta)}$ and $\xi_1\in \Reeb$ be a minimizer of $ \bar{\xi}\mapsto \BJ(\xi_0;\CF_\xi)$. Then $\BJ(\xi_0;\CF_{\xi_1}) =0$. Hence $\CF_{(\CX,\D_\CX,\xi_0;\eta+\xi_1)} = \CF_{\xi_1} = a\CF_{\wt_{\xi_0}}$ for some $a>0$ by Corollary \ref{Corollary: J(v)=0}. 
\end{proof}

\begin{cor}
\label{Corollary. reduced uniform to polystable}
If $(X,\D,\xi_0)$ is reduced uniformly K/Ding-stable for test configurations, then it is K/Ding-polystable for normal test configurations. 
\end{cor}

\begin{proof}
Let $(\CX, \D_\CX, \xi_0; \eta)$ be a normal test configurations of $(X,\D,\xi_0)$ with vanishing $\Fut$ or $\BD$. Since $\Fut$ or $\BD \ge \delta \BJ_\IT$ for some $\delta > 0$, we see that $\BJ_\IT(\CX, \D_\CX, \xi_0; \eta)= 0$. 
Hence it is a product test configuration by Lemma \ref{Corollary: J(CX)=0}. 
\end{proof}

\subsection{Approximating sequences}
Following \cite[Definition 3.55]{Xu24}, we may approximate a $\IN$-filtration $\CF$ by a sequence of finitely generated $\IN$-filtrations $\{\CF_m\}_{m\ge 1}$ as follows. 

\begin{defi}[Approximating sequences]\rm 
\label{Definition. approximating sequences}
Let $\CF$ be a linearly bounded $\IN$-filtration on $R$. 
We define a sequence $\{\CF_{m}\}$ of linearly bounded $\IN$-filtrations on $R$ as follows, which is called the {\it approximating sequence} of $\CF$. For any $m \ge 1$, let 
\begin{enumerate}
\item $\CF^\lam_m := \CF^\lam$ for $\lam\le m$, 
\item $\CF_m^{\lam} := \sum_{(\mu_1, \cdots, \mu_s)} \CF^{\mu_1} \cdots \CF^{\mu_s}$ for $\lam > m$, 
\end{enumerate}
where the sum runs over all $s\in \IN_{\ge1}$ and $(\mu_1, \cdots, \mu_s) \in \IN^{s}$ satisfying $\mu_i\le m$ and $\mu_1+ \cdots+ \mu_s \ge \lam$. 
\end{defi}

It's clear that $\CF_m\seq \CF$ is a linearly bounded $\IN$-filtration on $R$. 
By definition, the corresponding (non-extended) Rees algebra $\CR= \CR(\CF_m) := \oplus_{k\in\IN} \CF_m^k\cdot t^{-k}$ of $\CF_m$ is finitely generated over $\CR_0 = R$. Hence for sufficiently divisible $d_m\in m\IN$, we have 
\begin{eqnarray}
\label{Eqnarray. a_dl = a_d^l of approximating sequences}
\CF_m^{d_ml} = (\CF_m^{d_m})^l. 
\end{eqnarray} 
for any $l\in \IN$, by the following well-known lemma (see for example \cite[Lemma 10.56.2]{stacks-project}). 
\begin{lem}
Let $\CR=\oplus_{k\in \IN}\CR_k$ be a graded ring, which is finitely generated over $\CR_0$. Then for all sufficiently divisible $d$, we have $\CR_{ld} = \CR_d^l$ for any $l\in\IN$. 
\end{lem}

\begin{rmk}\rm
For any $\IN$-filtration $\CF$, we have $\Rees(\CF)/t\Rees(\CF) \cong \Gr_\CF R$. The extended Rees algebra $\Rees(\CF)$ is generated (over $\Ik$) by the sub-$\Ik$-algebra $\CR(\CF) \seq \Rees(\CF)$ and sub-$\Ik$-module $t\cdot R$. Hence if $\CR(\CF)$ is finitely generated over $\Ik$, then $\Rees(\CF)$ (equivalently, $\Gr_\CF R$) is finitely generated over $\Ik$. But the converse may not hold. 
\end{rmk}

\begin{cor}
\label{Corollary. lam_max convergence. approximating sequences}
Let $v_0\in\Val_{X,x}^*$ and $\CF$ be a linearly bounded $\IN$-filtration on $R$. Let $\{\CF_m\}$ be the approximating sequence of $\CF$. Then we have 
\begin{eqnarray}
\label{Eqnarray. convergence of lam_max(F_m), approximating sequence}
\mathop{\lim}_{m\to\infty} \lam_\max(v_0;\CF_m) = \lam_\max(v_0;\CF). 
\end{eqnarray}
\end{cor}
\begin{proof}
Since $\CF_m\seq \CF$, we have $\lam_\max(v_0;\CF_m)\le \lam_\max(v_0;\CF)$. 
By (\ref{Eqnarray. a_dl = a_d^l of approximating sequences}), we have 
$$\lam_\max(v_0;\CF_m) = \frac{1}{d_m} \lam_\max^{(d_m)}(v_0;\CF_m) \ge \frac{1}{m} \lam_\max^{(m)}(v_0;\CF_m) = \frac{1}{m} \lam_\max^{(m)}(v_0;\CF). $$
Taking $m\to \infty$, the proof is finished. 
\end{proof}

Let $(X,\D,\xi_0)$ be a log Fano cone with a good $\IT$-action. Then we also have the following convergence of $S$-invariants. 

\begin{thm}
\label{Theorem. S convergence. approximating sequences}
Let $\CF$ be a $\IT$-invariant linearly bounded $\IN$-filtration on $R$ and $\{\CF_m\}$ be the approximating sequence of $\CF$. Then we have 
\begin{eqnarray}
\label{Eqnarray. convergence of S(F_m), approximating sequence}
\mathop{\lim}_{m\to\infty} S(\xi_0 ;\CF_m) = S(\xi_0 ;\CF). 
\end{eqnarray}
\end{thm}

In order to prove the theorem, we need the following local version of \cite[Corollary 2.10]{BJ20}, which relies on the theory of Okounkov bodies of $R$. 
We will give a proof in the appendix for the reader's convenience. For any quasi-monomial valuation $v_0\in\Val_{X,x}^*$ and linearly bounded filtration $\CF$, we set $N_m=\ell(R/\CF_{v_0}^m)$ and
$$S''_m(v_0;\CF) := \frac{n+1}{n}\cdot\frac{\tS(v_0;\CF)}{mN_m}. $$
By (\ref{Eqnarray. N_m and tS(F0;F0)}), we have $\lim_{m\to \infty} \frac{\tS(v_0;v_0)}{m N_m} = \frac{n}{n+1}$. Hence $\lim_{m\to \infty} S''_m(v_0;\CF) = S(v_0;\CF)$.

\begin{lem}\label{Lemma. Blum-Jonsson uniform bound of S_m}
For any quasi-monomial valuation $v_0\in\Val_{X,x}^*$ and $\vep>0$, there exists an integer $m_0 = m_0(v_0,\vep) > 0$ such that 
\begin{eqnarray}
\label{Eqnarray. uniform estimate of S_m}
S''_m(v_0;\CF) \le (1+\vep)S(v_0;\CF)
\end{eqnarray}
for any linearly bounded filtration $\CF$ and any $m\ge m_0$. 
\end{lem}

\begin{proof}[Proof of Theorem \ref{Theorem. S convergence. approximating sequences}]
We follow the proof of \cite[Lemma 3.57]{Xu24}. 
Since $\CF_m\seq \CF$, we have $S(\xi_0 ;\CF_m)\le S(\xi_0 ;\CF)$. It suffices to show 
$$\underline{S} = \mathop{\liminf}_{m\to\infty}S(\xi_0 ;\CF_m)\ge S(\xi_0 ;\CF). $$
Since $\CF$ and $\CF_{\wt_{\xi_0}}$ are linearly bounded, there exists an integer $C>0$ such that $\CF^{Cm} \seq \CF_{\wt_{\xi_0}}^m$ for any $m\in\IN$. Hence 
$$\CF^k_{Cm}(R/\CF_{\wt_{\xi_0}}^m) = \CF^k(R/\CF_{\wt_{\xi_0}}^m), \quad \forall k\in \IN. $$
In particular, $S''_m(\xi_0 ;\CF_{Cm}) = S''_m(\xi_0 ;\CF)$. Pick a sequence of positive numbers $\vep_k\to 0$ as $k\to\infty$. By Lemma \ref{Lemma. Blum-Jonsson uniform bound of S_m}, there exists a sequence of positive integers $m_k=m_0(\wt_{\xi_0},\vep_k)$ such that 
$$S''_{m}(\xi_0 ;\CG) \le (1+\vep_k)S(\xi_0 ;\CG), \quad \forall m\ge m_k $$
for any linearly bounded filtration $\CG$. 
Let $n_k$ be a sequence of positive integers such that 
$$\mathop{\lim}_{k\to\infty}S(\xi_0 ;\CF_{Cn_k}) = \underline{S}.$$
By taking a subsequence, we may assume that $n_k\ge m_k$. Hence 
$$S''_{n_k}(\xi_0 ;\CF) = S''_{n_k}(\xi_0 ;\CF_{Cn_k}) \le (1+\vep_k)S(\xi_0 ;\CF_{Cn_k}). $$
Taking $k\to \infty$, we get $S(\xi_0;\CF) \le \underline{S}$. The proof is finished. 
\end{proof}

\begin{thm}
\label{Theorem. convergence of J_T, approximating sequence}
Let $\CF$ be a $\IT$-invariant linearly bounded $\IN$-filtration on $R$ and $\{\CF_m\}$ be the approximating sequence of $\CF$ (Definition \ref{Definition. approximating sequences}). Then $\CF_m$ is also $\IT$-invariant and $\{\CF_{m,\xi}\}$ is the approximating sequence of $\CF_\xi$ for any $\xi\in\Bt^+_\IR$. Moreover, we have 
\begin{eqnarray}
\label{Eqnarray. convergence of J_T(F_m), approximating sequence}
\mathop{\lim}_{m\to\infty} \BJ_\IT(\xi_0;\CF_m) = \BJ_\IT(\xi_0;\CF). 
\end{eqnarray}
\end{thm}

\begin{proof}
We follow the proof of \cite[Proposition 6.29]{Xu24}. The filtration $\CF_m$ is $\IT$-invariant since $\CF_m^k = \CF^k$ is $\IT$-invariant for any $0\le k\le m$, and $\CF_m^\lam$ is generated by these ideals for any $\lam\in\IN$. Since 
$$\CF_m^k 
= \CF^k 
= \oplus_{\alpha} \CF^k R_\alpha 
= \oplus_{\alpha} \CF_m^k R_\alpha, \quad 
\forall 0\le k\le m, $$
we have
$$\CF_{m,\xi}^k 
= \oplus_\alpha \CF_{m}^{k-\la\alpha,\xi\ra} R_\alpha
= \oplus_\alpha \CF^{k-\la\alpha,\xi\ra} R_\alpha
= \CF_\xi^k, \quad 
\forall 0\le k\le m,$$
where the second equality follows from $\xi\in\Bt_\IR^+$ (hence $\la\alpha, \xi\ra>0$). Checking directly we see that $\CF_{m,\xi}^\lam$ is generated by these ideals for any $\lam\in\IN$. Hence $\{\CF_{m,\xi}\}$ is the approximating sequence of $\CF_\xi$. 

We denote by $f, f_m: N_\IR' \to \IR$ the functions: 
\begin{eqnarray*}
\label{}
f(\bar{\xi}) = \BJ(\xi_0;\CF_\xi), \quad 
f_m(\bar{\xi}) = \BJ(\xi_0;\CF_{m,\xi}). 
\end{eqnarray*}
By Corollary \ref{Corollary. lam_max convergence. approximating sequences} and Theorem \ref{Theorem. S convergence. approximating sequences}, we have pointwise convergence: 
\begin{eqnarray}
\label{Eqnarray. pointwise convergence J(F_m,xi) to J(F_xi)}
\mathop{\lim}_{m\to\infty} f_m(\bar{\xi})= f(\bar{\xi})
\end{eqnarray}
for any $\xi\in\Bt_\IR^+$. 
Let $e = S(\xi_0;\CF) \ge S(\xi_0;\CF_m)$. Then by Lemma \ref{Lemma: J is convex w.r.t twist}, we have
\begin{eqnarray*}
\label{}
f(\bar{\xi}) \ge C|\bar{\xi}| - e, \quad 
f_m(\bar{\xi}) \ge C|\bar{\xi}| - e
\end{eqnarray*}
for any $\xi \in \Bt_\IR^+$, where $C$ is the Euclidean distance from $\alpha_0\in \interior(\BP)$ to $\partial \BP$ given by (\ref{Eqnarray. dist alpha0 to partial P}). Hence there exists a compact subset $K\seq N_\IR'$ such that the minimizers of $f$ and $f_m$ are all contained in $K$. On the other hand, we have 
\begin{eqnarray*}
\label{}
|f_m(\xi_1) - f_m(\xi_2)| \le 2C |\bar{\xi_1} - \bar{\xi_2}|
\end{eqnarray*}
for sufficiently large $m\in\IN$ and any $\xi_1,\xi_2\in\Bt_\IR^+$. 
Hence by the Arzel\`a-Ascoli theorem, the convergence (\ref{Eqnarray. pointwise convergence J(F_m,xi) to J(F_xi)}) is uniform on $K$. In particular, $\lim_{m\to\infty}$ commutes with $\min_{\bar{\xi}\in K}$: 
\begin{eqnarray*}
\label{}
\mathop{\lim}_{m\to\infty} \BJ_\IT(\xi_0;\CF_m) 
= \mathop{\lim}_{m\to\infty} \mathop{\min}_{\bar{\xi}\in K} f_m(\bar{\xi}) 
= \mathop{\min}_{\bar{\xi}\in K} \mathop{\lim}_{m\to\infty} f_m(\bar{\xi}) 
= \mathop{\min}_{\bar{\xi}\in K} f(\bar{\xi}) 
= \BJ_\IT(\xi_0;\CF). 
\end{eqnarray*}
\end{proof}

\begin{cor}
\label{Corollary. reduced uniform Ding for filtrations and finitely generated filtrations}
Let $(X,\D,\xi_0)$ be a log Fano cone singularity. Then it is reduced uniformly Ding-stable for filtrations if and only if it is reduced uniformly Ding-stable for finitely generated $\IN$-filtrations. 
\end{cor}
\begin{proof}
If $(X,\D,\xi_0)$ is not reduced uniformly Ding-stable for filtrations, then for any $\eta>0$, there exists a $\IT$-invariant linearly bounded $\IN$-filtration $\CF$ such that $\BD(\CF) < \eta \BJ_\IT(\xi_0;\CF)$. Let $\{\CF_m\}$ be the approximating sequence of $\CF$, whose members are finitely generated by (\ref{Eqnarray. a_dl = a_d^l of approximating sequences}). Then we have 
\begin{eqnarray*}
\label{}
\mathop{\lim}_{m\to\infty} \BD(\CF_m) = \BD(\CF), \quad 
\mathop{\lim}_{m\to\infty} \BJ_\IT(\xi_0;\CF_m) 
= \BJ_\IT(\xi_0;\CF)
\end{eqnarray*}
by Theorem \ref{Theorem. S convergence. approximating sequences} and Theorem \ref{Theorem. convergence of J_T, approximating sequence}. Hence $\BD(\CF_m) < \eta \BJ_\IT(\xi_0;\CF_m)$ for sufficiently large $m\in\IN$. 
\end{proof}

\begin{lem}
\label{Lemma. reduced uniform Ding. f.g. filtrations = weakly special divisors}
Let $(X,\D,\xi_0)$ be a log Fano cone singularity. If it is reduced uniformly Ding-stable for Koll\'ar components, then it is reduced uniformly Ding-stable for finitely generated filtrations. 
\end{lem}

\begin{proof}
By assumption, there exists $0<c<A(\xi_0)$ such that $\BD(v)-c\BJ_\IT(\xi_0; v) \ge 0$ for any $\IT$-invariant Koll\'ar component $v$. 
Since the product-type valuations $\wt_{\xi}$ for $\xi\in \Bt_\IR^+\cap N_\IQ$ are Koll\'ar components, we see that $\BD(\bar{\xi}) = \BD(\xi) = \BD(\wt_\xi) \ge 0$ for any $\xi\in \Bt_\IR^+\cap N_\IQ$. However, $\bar{\xi} \mapsto \BD(\bar{\xi})$ is linear on $N_\IR' \cong \Reeb/\la\xi_0\ra$, we see that $\BD(\xi)=0$ for any $\xi\in\Reeb$. In particular, $\BD(\CF_\xi)=\BD(\CF)+\BD(\xi) = \BD(\CF)$ for any linearly bounded filtration $\CF$ and any $\xi \in \Reeb$. 

Let $\CF$ be a $\IT$-invariant finitely generated $\IN$-filtration, that is, there exists $d \in \IN$ such that $(\CF^{d})^k = \CF^{kd}$ for any $k\in\IN$. Then by \cite[Lemma 1]{Xu14}, there exists a $\IT$-invariant Koll\'ar component $v$ minimizing $\lct(\CF) = d\cdot\lct(\CF^{d})$. So we have $\CF\seq \CF_v$ and $\lct(\CF) = \lct(\CF_v)$, which is preserved by $\xi$-twist for any $\xi\in\Reeb$. Hence 
\begin{eqnarray*}
S(\xi_0;\CF_\xi) \le S(\xi_0;\CF_{v,\xi}), \quad
\lam_\max (\xi_0;\CF_\xi) \le \lam_\max (\xi_0;\CF_{v,\xi}). 
\end{eqnarray*}
Then  
\begin{eqnarray*}
& & \lct(\CF_\xi) - (A(\xi_0)-c)S(\xi_0;\CF_\xi) -c\lam_\max (\xi_0;\CF_\xi)\\
&\ge& \lct(\CF_{v,\xi}) - (A(\xi_0)-c)S(\xi_0;\CF_{v,\xi}) -c\lam_\max (\xi_0;\CF_{v,\xi}). 
\end{eqnarray*}
In other words, 
$$\BD(\CF_\xi) - c\BJ(\xi_0;\CF_\xi) \ge \BD(\CF_{v,\xi}) - c\BJ(\xi_0;\CF_{v,\xi}). $$
Since $\BD(\CF_\xi)=\BD(\CF)$ and $\BD(\CF_{v,\xi})=\BD(\CF_v)$ are independent of $\xi$, we conclude by taking supremum for all $\xi\in \Bt_\IR^+$: 
$\BD(\CF) - c\BJ_\IT(\xi_0;\CF) \ge \BD(\CF_{v}) - c\BJ_\IT(\xi_0;\CF_{v})\ge 0. $
\end{proof}

We are ready to prove the second main theorem of the paper, that is, Ding-polystability for special test configurations implies reduced uniform Ding-stability for filtrations.  

\begin{proof}[Proof of Theorem \ref{Theorem: Intro. Ding-ps for special test configurations implies reduced uniform Ding}]
By Corollary \ref{Corollary. reduced uniform Ding for filtrations and finitely generated filtrations}, assume that $(X,\D,\xi_0)$ is a Ding-semistable log Fano cone singularity ($\delta_\IT(X,\D,\xi_0)= 1$), which is not $\IT$-equivariantly reduced uniformly Ding-stable for finitely generated filtrations. Then by Lemma \ref{Lemma. reduced uniform Ding. f.g. filtrations = weakly special divisors}, there exists a sequence of Koll\'ar components $v_i\in\Val^{\IT,*}_{X,x}$ (not of the form $\wt_\xi$), such that $\BD(v_i)/\BJ_\IT(\xi_0;v_i)\to 0$ as $i\to \infty$. Rescale $\xi_0$ and $v_i$ such that $A(\xi_0)=A_{X,\D}(v_i)=1$. Since $(X,\D,\xi_0)$ is Ding-semistable, we see that 
\begin{eqnarray*}
1\le \frac{A_{X,\D}(v_{i,b\xi_0})}{S(\xi_0; v_{i,b\xi_{0}})} 
= \frac{A_{X,\D}(v_i)+b}{S(\xi_0; v_i)+b}
\le \frac{A_{X,\D}(v_i)}{S(\xi_0; v_i)}
\end{eqnarray*}
for any $b>0$. By definition of the reduced J-norm, we also have $\BJ_\IT(\xi_0;v_{i,b\xi_0}) = \BJ_\IT(\xi_0;v_{i})$. Replacing $v_i$ by $\frac{1}{2}v_{i,\xi_0}$, we may assume that $v_i(\CF_{\wt_{\xi_0}})>\frac{1}{2}$ and $A_{X,\D}(v_i)=1$. Since $\CF_{\wt_{\xi_0}}$ is linearly bounded by $\fm_x$, there exists $\vep_0>0$ such that $v_i(\fm_x)>\vep_0$ for any $i$. 

By \cite[Theorem 3.3]{Xu19}, there exists an $N$-complement $\Gamma$ of $(X,\D)$ such that $v_i\in \LC(X,\D+\Gamma)$, and (by taking subsequnce) $v=\lim_{i\to \infty} v_i \in \Val_{X,x}^{\IT,*}$ exists in some simplicial cone $\sigma\seq \LC(X,\D+\Gamma)$ which is not of product type by the argument of \cite[Proof of Theorem A.5]{XZ19}. By Lemma \ref{Lemma. lam_max is continuous on quasi-monomial cone} and Corollary \ref{Corollary. S is concave on quasi-monomial cone}, the functions $\lam_\max(\xi_0;-)$ and $S(\xi_0;-)$ are continuous on $\sigma$. Hence 
$$\mathop{\limsup}_{i\to \infty} \BJ_\IT(\xi_0;v_i) \le  \mathop{\lim}_{i\to\infty}\BJ(\xi_0;v_i) = \BJ(\xi_0;v) < +\infty. $$ 
Thus $\BD(v)=\lim_{i\to \infty} \BD(v_i)=0$ and $v$ minimizes $\delta_\IT(X,\D,\xi_0)=1$. 

Finally, since $v$ minimizes $\delta_\IT(X,\D,\xi_0)=1$, it is a Koll\'ar valuation by Theorem \ref{Theorem. Local optimal destablization}. We can perturb $v$ and get a Koll\'ar component $E$ minimizing $\delta_\IT(X,\D,\xi_0)=1$ also by Theorem \ref{Theorem. Local optimal destablization}, which is not of product type since this holds for $v$. It induces a non-product special test configuration of $(X,\D,\xi_0)$ with vanishing $\BD$. We conclude that $(X,\D,\xi_0)$ is not $\IT$-equivariantly Ding-polystable for special test configurations. 
\end{proof}

\begin{proof}[Proof of Corollary \ref{Corollary. Intro. all equivalent}]
The main strategy was explained in the introduction. By Corollary \ref{Corollary. Fut ge Ding}, we get $(4)\Rightarrow (2) \Rightarrow (1) \Leftrightarrow (3)$ and $(8)\Rightarrow (6) \Rightarrow (5) \Leftrightarrow (7)$. We have $(8)\Rightarrow (4)$ by Corollary \ref{Corollary. reduced uniform to polystable}. The implication $(9)\Rightarrow (8)$ is clear since every test configuration corresponds to a linearly bounded filtration. We conclude by showing $(3)\Rightarrow (9)$ in Theorem \ref{Theorem: Intro. Ding-ps for special test configurations implies reduced uniform Ding}. 
\end{proof}

\subsection{Reduced delta invariants}
In the following, we assume that $\Fut|_N=0$, which means that for any $\xi\in \Reeb$, we have $A(\xi)=A(\xi_0)S(\xi_0;\xi)$. By linear extension, this also holds for any $\xi\in\IN_{\IR}$. 
We define the {\it reduced delta invariant} of the log Fano cone $(X,\D,\xi_0)$ by 
\begin{eqnarray}
\delta^\red_{\IT}(X,\D,\xi_0) := \mathop{\inf}_{v\in \Val^{\IT,*}_{X,x} \setminus \Reeb} \mathop{\sup}_\xi  \frac{A_{X,\D}(v_\xi)}{A(\xi_0)S(\xi_0;v_\xi)}, 
\end{eqnarray} 
where the supremum runs over all $\xi\in N_\IR$ such that $v_\xi \in \Val_{X,x}^{\IT,*}$. 

\begin{lem}
\label{Lemma. delta_red ge 1}
We always have $\delta^\red_{\IT}(X,\D,\xi_0)\ge 1$. 
\end{lem}
\begin{proof}
If there exists $v\in \Val^{\IT,*}_{X,x} $ such that $\frac{A_{X,\D}(v)}{A(\xi_0)S(\xi_0;v)} <1$, then \begin{eqnarray*} \frac{A_{X,\D}(v_{a\xi_0})}{A(\xi_0)S(\xi_0;v_{a\xi_0})} = \frac{A_{X,\D}(v)+aA(\xi_0)}{A(\xi_0)S(\xi_0;v)+aA(\xi_0)} \to 1, \quad a\to +\infty.  \end{eqnarray*} 
\end{proof}

\begin{rmk}\rm
If we do not assume that $\Fut|_N=0$ in the definition of $\delta^\red_\IT$, then using the argument in the above lemma we will see that $\delta^\red_\IT=+\infty$ whenever $\Fut|_N \ne 0$, which is not well-behaved. 
\end{rmk}

The following theorem says that the reduced uniform stability of a log Fano cone $(X,\D,\xi_0)$ (satisfying $\Fut|_N=0$) can be characterized by $\delta^\red_\IT(X,\D,\xi_0)$. 

\begin{thm}
\label{Lemma. characterization of delta^red > 1}
The following statements are equivalent. 
\begin{enumerate}
\item $(X,\D,\xi_0)$ is reduced uniformly Ding-stable for filtrations; 
\item $\delta^\red_{\IT}(X,\D,\xi_0)>1$; 
\item $\delta_\IT(X,\D,\xi_0)= 1$ and any minimizer $v\in\Val_{X,x}^{\IT,*}$ of $\delta_\IT(X,\D,\xi_0)$ is of product type; 
\item $(X,\D,\xi_0)$ is Ding-polystable for special test configurations. 
\end{enumerate}
\end{thm}

\begin{proof}
For $(1) \Rightarrow (2)$, 
there exists $\eta>0$ such that $\BD(\CF)\ge \eta\cdot \BJ_\IT(\CF)$ for any $\IT$-invariant linearly bounded filtration $\CF$. In particular $(X,\D,\xi_0)$ is Ding-semistable, hence $\Fut|_N=0$. For any $v\in \Val^{\IT,*}_{X,x}$, we have 
\begin{eqnarray*}
\mathop{\sup}_\xi  \frac{A_{X,\D}(v_\xi)}{A(\xi_0)S(\xi_0;v_\xi)} 
&=& 1+\mathop{\sup}_\xi  \frac{\BD(v_\xi)}{A(\xi_0)S(\xi_0;v_\xi)} 
= 1+\mathop{\sup}_\xi \frac{\BD(v)}{A(\xi_0)S(\xi_0;v_\xi)}  \\
&=& 1+\frac{\BD(v)}{A(\xi_0)\mathop{\inf}_\xi S(\xi_0;v_\xi)} \\
&\ge& 1+\frac{\BD(v)}{(n-1)A(\xi_0) \BJ_\IT(\xi_0;v)} 
\ge 1+ \frac{\eta}{(n-1)A(\xi_0)}, 
\end{eqnarray*}
where the first inequality follows from Theorem \ref{Corollary. S_T bounded by J_T}. 

Next we prove $(2) \Rightarrow (3)$. Since $\Fut|_N=0$, we see that if $\delta^\red_{\IT}(X,\D,\xi_0)>1$, then for any $v\in\Val_{X,x}^{\IT,*}\setminus \Reeb$, there exists $\xi\in N_\IR$ such that
\begin{eqnarray*} \frac{A_{X,\D}(v)+A(\xi)}{A(\xi_0)S(\xi_0;v)+A(\xi)}=\frac{A_{X,\D}(v_{\xi})}{A(\xi_0)S(\xi_0;v_{\xi})} >1. \end{eqnarray*} 
Hence $\frac{A_{X,\D}(v)}{A(\xi_0)S(\xi_0;v)} >1$. 

Then we prove $(3) \Rightarrow (4)$. Assume that $(X,\D,\xi_0)$ is not Ding-polystable for special test configurations. Then there exists a non-product type special test configuration $(\CX,\D_\CX,\xi_0;\eta)$ satisfying $\BD(\CX,\D_\CX,\xi_0;\eta)= 0$. Let $\CF$ be the corresponding filtration, then $\CF=c\CF_v$ for some $c>0$ and some non-product type Koll\'ar component $v$. Hence $v$ is a non-product type minimizer of $\delta_\IT(X,\D,\xi_0)$. We get a contradiction. 

Finally, $(4) \Rightarrow (1)$ follows from Theorem \ref{Theorem: Intro. Ding-ps for special test configurations implies reduced uniform Ding}. 
\end{proof}


\begin{cor}
Let $(X,\D,\xi_0)$ be a log Fano cone singularity satisfying $\delta_\IT(X,\D,\xi_0)= 1$. If $\delta^{\red}_\IT(X,\D,\xi_0) = 1$, then there exists a $\IT$-invariant Koll\'ar component $E$ over $(X,\D)$ such that $A_{X,\D}(E)=A(\xi_0)S(\xi_0;E)$ and $\ord_E$ is not of the form $\wt_\xi$ for any $\xi\in\Bt_\IR^+$. 
\end{cor}
\begin{proof}
By Lemma \ref{Lemma. characterization of delta^red > 1}, the log Fano cone is not reduced uniformly Ding-stable for filtrations. Hence by the proof of Theorem \ref{Theorem: Intro. Ding-ps for special test configurations implies reduced uniform Ding}, there exists a Koll\'ar component $E$ not of the form $\wt_\xi$ for any $\xi \in \Reeb$ minimizing $\delta_\IT(X,\D,\xi_0)= 1$. The proof is finished. 
\end{proof}

\appendix

\section{Okounkov bodies}

In this appendix, we aim to prove Lemma \ref{Lemma. Blum-Jonsson uniform bound of S_m} using the theory of Okounkov bodies developed by \cite{KK14,BLQ22}. 

Let $R$ be an integral domain over $\Ik$. Equip the group $\IZ^n$ with a total order ``$\le$'' respecting addition, that is, $a\le b \Rightarrow a+c \le b+c$ for any $a,b,c\in \IZ^n$. For any $\Ik$-vector space $V$, we denote by 
$$V^*=V\setminus \{0\}. $$
A {\it $\IZ^n$-valuation} $\fv: R^* \to \IZ^n$ is a function satisfying:  
\begin{enumerate}
\item $c\in \Ik^* \Rightarrow \fv(c) = 0$;
\item $f,g\in R^* \Rightarrow \fv(fg)=\fv(f)+\fv(g)$;
\item $f,g\in R^* \Rightarrow \fv(f+g)\ge \min\{\fv(f),\fv(g)\}$.
\end{enumerate}
We say that $\fv$ has {\it one-dimensional leaves} if, whenever $\fv(f) = \fv(g)$, there exists $c\in \Ik$ such that $\fv(g + cf) > \fv(g)$. 
By definition $\CS = \fv(R^*) \cup \{0\}$ is an additive sub-semigroup of $\IZ^n$, which is called the {\it value
semigroup} of $(R, \fv)$. We denote by $\Cone(\CS)\seq \IR^n$ the closed convex cone spanned by $\CS$. It is called {\it strongly convex} if there exists a linear function $\ell: \IR^n \to\IR$ such that $\ell^{-1}(0)\cap \Cone(\CS) = \{0\}$. 

\begin{defi}[Good valuations]\rm
\label{Definition. good valuations}
Let $(R,\fm)$ be the local domain of a closed point $x\in X$ in an $n$-dimensional normal variety $X$ over $\Ik$. We say that a $\IZ^n$-valuation $\fv$ on $R$ is {\it good} if
\begin{enumerate}
\item $\fv$ has one-dimensional leaves; 
\item $\CS$ generates $\IZ^n$ as a group; 
\item $\Cone(\CS)\seq \IR^n$ is strongly convex, that is, there exists a linear function $\ell: \IR^n \to\IR$ such that $$\Cone(\CS)\seq \ell_{>0};$$
\item There exist $r_0 >0$ and a linear function $\ell: \IR^n \to\IR$ such that, for any $f\in R\setminus \{0\}$, we have
$$\ord_\fm(f) \ge r_0 \cdot \ell(\fv(f)). $$  
\end{enumerate}
\end{defi}

In the following, let $x\in(X = \Spec(R),\D)$ be a klt singularity. 

\begin{lem} \cite[Lemma A.2]{BLQ22}
\label{Lemma. BLQ good valuation}
Let $v_0\in\Val_{X,x}^*$ be a quasi-monomial valuation. Then there exists a vector $\Balpha_0\in \IR_{\ge 0}^n$ and a good valuation $\fv$ on $R$ such that 
$$v_0(f) = \la\Balpha_0, \fv(f)\ra, \quad \forall f\in R^*. $$
Moreover, we can take $\ell=\la\Balpha_0,-\ra$ for condition (3) and (4) of a good valuation. 
\end{lem}

We fix the quasi-monomial valuation $v_0\in\Val_{X,x}^*$ and let $\CF_0 = \CF_{v_0}$. Let $\Balpha_0 = (\alpha_1,\cdots,\alpha_r,0,\cdots,0) \in \IR_{\ge 0}^n$ and the good valuation $\fv$ be as in Lemma \ref{Lemma. BLQ good valuation}. 
\begin{lem}
\label{Lemma. Semigroup given by filtration}
For any linearly bounded filtration $\CF$, $l_0\in \IZ_{\ge1}$ and $\lam,t\ge 0$, we denote by 
\begin{eqnarray*}
\label{}
\Gamma_{m,\lam}(\CF) &:=&  \fv(\CF^{\lam}R^*)\setminus \fv(\CF_0^{>m}R^*) \seq \IN^n,  \\
\Gamma(\CF^{(t)}) &:=& \{(\Bbeta, m)\in \IN^n\times l_0\IN \mid \Bbeta \in \Gamma_{m,mt}(\CF)\} \cup \{(0,0)\}. 
\end{eqnarray*}
Then there exists $l_0\in\IZ_{\ge 1}$ such that the sub-semigroup $\Gamma(\CF^{(t)})\seq \IN^n\times l_0\IN$ satisfies: 
\begin{enumerate}
\item $\Gamma(\CF^{(t)}) \cap (\IN^n\times\{0\}) = \{(0,0)\}$; 
\item $\Gamma(\CF^{(t)})$ is contained in a sub-semigroup $B\seq \IN^n\times l_0\IN$ which is generated by finitely many $(v_i,l_0)\in \IN^n\times l_0\IN$; 
\item $\Gamma(\CF^{(t)})\seq \IN^n\times l_0\IN$ generates $\IZ^n\times l_0\IZ$ as a group. 
\end{enumerate}
\end{lem}

\begin{proof}
For any $(\Bbeta,m), (\Bbeta',m') \in \Gamma(\CF^{(t)})$ with $m,m'>0$, there exist $s\in \CF^{mt}, s' \in \CF^{m't}$ such that $v_0(s)\le  m, v_0(s')\le m'$ and $\Bbeta=\fv(s), \Bbeta'=\fv(s')$. Hence $v_0(ss')\le m+m'$, $ss'\in\CF^{(m+m')t}$ and $\fv(ss') =\Bbeta+\Bbeta'$. In particular, $(\Bbeta+\Bbeta',m+m') \in \Gamma(\CF^{(t)})$. Hence $\Gamma(\CF^{(t)})$ is a semigroup. 

Since $\Gamma_{0,0}(\CF)=\fv(\CF^0 R^*)\setminus \fv(\CF_0^{>0}R^*) = \{0\}$, we see that condition (1) holds. 
Recall that $\Cone(\CS)$ is strongly convex and $\Cone(\CS) \seq \ell_{>0}$ for $\ell=\la\Balpha_0,-\ra$. Then we may enlarge $\CS$ and get a finitely generated semigroup $\widetilde{\CS} \seq \IN^n\cap \ell_{>0}$. There exists $l_0\in\IN$ such that $\widetilde{\CS}$ is generated by the finite subset $\widetilde{\CS}\cap \ell_{\le l_0}$. Then we may choose a linear function $\tell = \la\tilde{\Balpha},-\ra$ with $\tilde{\Balpha}=(\talpha_1,\cdots,\talpha_n)\in\IR^n_{>0}$ (where $\talpha_i>0$ are small enough) such that 
\begin{eqnarray*}
\widetilde{\CS}\cap \ell_{\le l_0} \seq \IN^n\cap \tell_{\le l_0}. 
\end{eqnarray*}
Hence
\begin{eqnarray*}
\Gamma(R) &=& \{(\Bbeta, m)\in \IN^n\times l_0\IN \mid \Bbeta = \fv(s), s\in R^*, v_0(s) \le m \} \\
&=& \{(\Bbeta, m)\in \IN^n\times l_0\IN \mid \Bbeta \in \CS, \ell(\Bbeta) \le m \}\\ 
&\seq& \{(\Bbeta, m)\in \IN^n\times l_0\IN \mid \Bbeta \in \tCS, \ell(\Bbeta) \le m \}\\ 
&\seq& \{(\Bbeta, m) \in \IN^n\times l_0\IN \mid \tell(\Bbeta) \le m\} =:B.  
\end{eqnarray*}
Then $B\seq \IN^n\times l_0\IN$ is a sub-semigroup generated by $\{(\Bbeta,l_0)\mid \Bbeta \in \IN^n\cap \tell_{\le l_0}\}$, which is a finite set. 
We see that condition (2) holds since $\Gamma(\CF^{(t)}) \seq \Gamma(\CF^{(0)})=\Gamma(R) \seq B$. 

Since $\fv:R^*\to \IZ^n$ is a good valuation, there exist $f_i$ in the fractional field of $R$ such that $\fv(f_i) = e_i$, where $\{e_1, \cdots, e_n\} \seq \IN^{n}$ is the standard basis. We also set $f_0 = 1$ and $e_0 = 0 \in \IN^n$. Pick $s_0\in R$ such that $f_is_0 \in R$ for any $1\le i\le n$. Let $\Bbeta_0 = \fv(s_0)$ and 
$$m_0 = \max\{ \lceil f_iv_0(s_0)\rceil \mid 0\le i\le n \} + 1. $$
Then $v_0(f_is_0)< m_0$ and $\fv(f_is_0) = e_i +\Bbeta_0$ for any $0\le i\le n$, that is, $(e_i+\Bbeta_0, m_0) \in \Gamma(\CF^{(t)})$. Moreover, $v_0(s_0)< m_0<m_0+1$ implies that $(\Bbeta_0, m_0+1) \in \Gamma(\CF^{(t)})$. Hence, the sub-group of $\IZ^n\times l_0\IZ$ generated by $\Gamma(\CF^{(t)})$ contains $(e_i, 0), 1\le i\le n$ and $(0,1)$, that is, the whole $\IZ^n\times l_0\IZ$. So condition (3) holds. 
\end{proof}

Let $C(\CF^{(t)}) \seq \IR_{\ge 0}^n\times\IR_{\ge 0}$ be the closed convex cone spanned by $\Gamma(\CF^{(t)})$. Then 
$$\BO(\CF^{(t)}) = C(\CF^{(t)}) \cap (\IR_{\ge 0}^n\times\{1\}) \seq \IR_{\ge 0}^n $$
is a closed convex body, which is called the {\it Okounkov body} of $\CF^{(t)}$ with respect to the quasi-monomial valuation $v_0$ (see \cite[Section 4]{HMQWZ}). We also have 
\begin{eqnarray*}
\BO(\CF^{(t)}) = \overline{\bigcup_{m\in\IN} \BO_{m,mt}(\CF)},\quad 
\BO_{m,mt}(\CF) = \frac{1}{m} \Gamma_{m,mt}(\CF). 
\end{eqnarray*}
If $t=0$ or $\CF=\CF_{\triv}$, we simply set
\begin{eqnarray*}
\Gamma_m = \fv(R^*)\setminus \fv(\CF_0^{m}R^*), \quad
\BO_m = \frac{1}{m}\Gamma_m, \quad
\BO = \overline{\bigcup_{m\in\IN} \BO_m}, 
\end{eqnarray*}
where $\BO$ is called the Okounkov body of $R$ with respect to $v_0$. 

\begin{cor}\cite[Proposition 2.1]{LM09}
The following limit exists 
\begin{eqnarray*}
\mathop{\lim}_{m\to \infty} \frac{\#(\BO_{m,mt}(\CF))}{m^n} = \vol(\BO(\CF^{(t)})). 
\end{eqnarray*}
\end{cor}

Note that for any $t\le t'$, we have $\BO(\CF^{(t)}) \supseteq \BO(\CF^{(t')})$. We define the {\it concave transform} of $\CF$ by 
\begin{eqnarray*}
G_\CF:\IR^n\to \IR\cup\{-\infty\}, \quad 
G_\CF(\By):= \sup\{t\in \IR_{\ge 0}\mid \By\in\BO(\CF^{(t)})\}.
\end{eqnarray*}
By definition, if $\By\in\IR^n\setminus \BO$, then $G_\CF(\By)=-\infty$; if $t\in\IR$, then 
\[
\{\By\in \IR^n\mid G_\CF(\By) \ge t\} = \BO(\CF^{(t)}), 
\]
which is a closed convex subset of $\IR^n$, hence $G_\CF:\IR^n \to\IR \cup \{-\infty\}$ is concave and upper-semicontinuous. In particular, $G_\CF$ is locally Lipschitz continuous on $\interior(\BO)$. 
By Lemma \ref{Lemma. BLQ good valuation}, we see that $G_{v_0}(\By) = \la \Balpha_0, \By \ra$ for some $\Balpha_0 \in \IR^n_{\ge 0}$. Hence 
\begin{eqnarray}
\BO = \{\By\in \Cone(\CS) \mid \la \Balpha_0, \By \ra \le 1\}. 
\end{eqnarray}
The asymptotic invariants used throughout the paper can be reformulated using Okounkov bodies: 
\begin{eqnarray*}
\mult(\CF_{0}) 
&=& \vol(v_0) 
\,\,\,=\,\,\, \vol(R/\CF_{0}^{(1)}) 
\,\,\,=\,\,\, n!\cdot \vol(\BO), \\
\mult(\CF_0;t^{-1}\CF) 
&=& \vol(\CF^{(t)}(R/\CF_{0}^{(1)})) 
\,\,\,=\,\,\, n! \cdot \vol(\BO(\CF^{(t)})), \\
\frac{n}{n+1}S(\CF_0;\CF) 
&=& \int_0^{\lam_\max(\CF_0;\CF)} \frac{\mult(\CF_0;t^{-1}\CF)}{\mult(\CF_0)} \dif t
\,\,\,=\,\,\, \int_0^{\lam_\max(\CF_0;\CF)} \frac{\vol(\CF^{(t)}(R/\CF_{0}^{(1)}))}{\vol(R/\CF_{0}^{(1)})} \dif t \\
&=& \int_0^{\lam_\max(\CF_0;\CF)} \frac{\vol(\BO(\CF^{(t)}))}{\vol(\BO)} \dif t 
\,\,\,=\,\,\, \int_0^{\lam_\max(\CF_0;\CF)} \frac{\vol(\BO : G_\CF\ge t)}{\vol(\BO)} \dif t \\
&=& \frac{1}{\vol(\BO)}\int_\BO G_\CF(\By) \LE(\dif\By), 
\end{eqnarray*}
where $\LE(\dif \By)$ is the Lebesgue measure on the Okounkov body $\BO$. 

\begin{proof}[Proof of Lemma \ref{Lemma. Blum-Jonsson uniform bound of S_m}]
Let $\CF_0=\CF_{v_0}$ and $N_m = \ell(R/\CF_{0}^m)$. Let $\{s_j\}_{j=1}^{N_m} \seq R$ be a sequence compatible with both $\CF_0$ and $\CF$ whose image in $R/\CF_{0}^m$ is a basis. Then 
\begin{eqnarray*}
\tS_m(\CF_0;\CF) 
= \sum_{j=1}^{N_m} \ord_\CF(s_j) 
&\le& m\cdot \sum_{j=1}^{N_m} G_\CF(\fv(s_j)) \\
&=& m\cdot \sum_{\By \in \BO_m} G_\CF(\By) 
\le m^{n+1} \int_{\BO} G_\CF(\By) \LE_m(\dif \By).
\end{eqnarray*}
Since $N_m/m^n \to \vol(\BO)$ as $m\to \infty$, there exists $m_1\in \IN$ such that 
$$\vol(\BO)\le (1+\frac{\vep}{2})\frac{N_m}{m^n}$$ for any $m\ge m_1$. Let 
$$\vep' = \frac{\vep \cdot \vol(\BO)}{(n+1)(2+\vep)}>0. $$
By \cite[Lemma 2.2]{BJ20}, there exists $m_2 = m_0(\vep')\in\IN$ such that 
\begin{eqnarray*}
\int_{\BO} G(\By) \LE_m(\dif \By) \le \int_{\BO} G(\By) \LE(\dif \By) + \vep' 
\end{eqnarray*}
for any concave function $G:\BO\to \IR$ with $0\le G\le 1$, and for any $m\ge m_2$. 
For any linearly bounded filtration $\CF$, we set $G= G_\CF/\lam_{\max}(\CF_0;\CF)$. Hence 
\begin{eqnarray*}
S''_m(\CF_0;\CF) 
&=& \frac{n+1}{n}\cdot\frac{\tS(\CF_0;\CF)}{mN_m} 
\,\,\,\le\,\,\, \frac{n+1}{n} \frac{1}{N_m/m^n} \int_{\BO} G_\CF(\By) \LE_m(\dif \By) \\
&\le& \frac{n+1}{n} \frac{1}{N_m/m^n} \Big(\int_{\BO} G_\CF(\By) \LE(\dif \By) + \vep'\lam_{\max}(\CF_0;\CF)\Big) \\
&=&  \frac{\vol(\BO)}{N_m/m^n} S(\CF_0;\CF) + \frac{n+1}{n}\frac{\vep'}{N_m/m^n} \cdot \lam_{\max}(\CF_0;\CF)  \\
&\le&  (1+\frac{\vep}{2}) S(\CF_0;\CF) + \frac{\vep}{2n} \cdot n S(\CF_0;\CF) 
\,\,\,=\,\,\, (1+\vep) S(\CF_0;\CF)
\end{eqnarray*}
for any linearly bounded filtration $\CF$ and for any $m\ge m_0= \max\{m_1,m_2\}$, where the third inequality follows from left hand side of (\ref{Eqnarray. S - lam_min bounded by lam_max - lam_min}). 
\end{proof}

The rest of this appendix is of independent interest. 
For a log Fano cone singularity $(X = \Spec(R),\D,\xi_0)$ with a good $\IT$-action, we can construct Okounkov bodies with additional nice properties. 
The following lemma gives a good valuation $\fv$ on the log Fano cone such that Lemma \ref{Lemma. BLQ good valuation} holds for many product-type valuations $v_0$. 

\begin{lem}
\label{Lemma. Appendix. good valuation cpt. torus action}
Let $\xi\in \Bt_\IR^+$ be a Reeb vector of rational rank $r$. Then there exist $\IQ$-linearly independent primitive vectors $\xi_1,\cdots,\xi_r\in N\cap \Bt_\IR^+$, whose spanning cone contains $\xi$, and a good valuation $\fv: R\setminus \{0\} \to \IZ^n$ satisfying $\pr_i(\fv(f)) = \la\alpha, \xi_i\ra$ for any $f\in R_\alpha$ and $1\le i\le r$. 

In particular, for any $\xi'= a_1\xi_1+\cdots+ a_r\xi_r$ with $a_i>0$, we have $\wt_{\xi'}(f) = \ell_{\xi'}(\fv(f))$ for any $f\in R$, where  $\ell_{\xi'}(\beta_1,\cdots,\beta_n) = a_1\beta_1+\cdots+a_r\beta_r$. 

We say that $\fv$ is a {\rm $\IT$-equivariant good valuation}. 
\end{lem}

\begin{proof}
Since $\xi\in\Bt_\IR^+$ is of rational rank $r$, there exist $\IQ$-linearly independent primitive vectors $\xi_1,\cdots,\xi_r\in N\cap \Bt_\IR^+$, whose spanning cone containing $\xi$. 
By \cite[Theorem 3.14]{XZ22} (see also \cite[Theorem 6.1]{Wang24}), there exists a dlt Fano type model $(Y,E=E_1+ \cdots + E_r) \to (X,\D)$ such that $\ord_{E_i} = \wt_{\xi_i}$. We may further choose prime divisors $E_{r+1},\cdots, E_{n} \seq Y$ such that $(Y, E_1+\cdots+E_n)$ is log smooth at some closed point $y\in E_1\cap\cdots\cap E_n$. Let $z_1,\cdots, z_n\in \CO_{Y,y}$ be the regular functions such that $E_i = \{z_i=0\}$ on some neighbourhood of $y$. We define the following function: 
$$\fv: \widehat{\CO}_{Y,y} \setminus \{0\} \to \IZ^n, \quad \fv(f) = \min\{\beta \in \IN^n \mid f_\beta \ne 0\},$$
where $f=\sum_{\beta} f_\beta z^\beta = \sum_{\beta=(\beta_1,\cdots,\beta_n)\in \IN^n} f_\beta \cdot z_1^{\beta_1}\cdots z_n^{\beta_n}$, and $\IZ^n$ is equipped with the lexicographic order. Then $\fv$ is a $\IZ^n$-valuation centered at the closed point $y\in Y$. In particular, it has one-dimensional leaves and Definition \ref{Definition. good valuations} (1) holds. 
As shown in [Cut13, Section 4], there exists $\beta'\in \CS$ such that $\beta'+e_i \in \CS$ for $1\le i\le n$, where $e_i \in \IN^n$ is the $i$-th standard basis vector. Hence $\CS$ generates $\IZ^n$ as a group and Definition \ref{Definition. good valuations} (2) holds. 

Moreover, for any $f\in R_\alpha$, we have $\ord_{E_i}(f) = \wt_{\xi_i}(f) = \la\alpha,\xi_i\ra$ for any $1\le i\le r$. Hence $f= z_1^{\la\alpha,\xi_1\ra}\cdots z_r^{\la\alpha,\xi_r\ra} \tilde{f}$ for some $\tilde{f}\in \CO_{Y,y}$ with $\ord_{E_i}(\tilde{f}) = 0$ for any $1\le i\le r$. In particular, $\pr_i(\fv(f)) = \la\alpha, \xi_i\ra$ for any $1\le i\le r$. 

Let $\xi'= a_1\xi_1+\cdots+ a_r\xi_r$ for some $a_i>0$. We choose the linear function $\ell_{\xi'}(\beta_1,\cdots,\beta_n)=a_1\beta_1+\cdots+ a_r\beta_r$. Then for any $f\in R_\alpha$, we have 
$$\wt_{\xi'}(f) = \la\alpha,\xi'\ra = a_1\la\alpha,\xi_1\ra +\cdots +a_r\la\alpha,\xi_r\ra = \ell_{\xi'}(\fv(f)).$$ 
Now the $\IZ^n$-valuation $\fv$ and the linear function $\ell_{\xi'}$ satisfy condition (3) and (4) of Definition \ref{Definition. good valuations}. Indeed, for any $f\in R$ with $\fv(f) \ne 0$, the decomposition $f=\sum_\alpha f_\alpha, f_\alpha \in R_\alpha$ satisfies $f_0=0$, since the $\IT$-action is good and $R_0 = \Ik$. Hence $\ell_{\xi'}(\fv(f)) \ge \min\{a_i\mid 1\le i\le r\} > 0$. We see that $\CS \seq (\ell_{\xi'})_{>0}$ and Definition \ref{Definition. good valuations} (3) holds. By the Izumi-type inequality \cite{Li18}, there exists a constant $c>0$ such that $\ell_{\xi'}(\fv(f))=\wt_{\xi'}(f)\le c\cdot \ord_{\fm_x}(f)$ for any $f\in R\setminus \{0\}$. Hence Definition \ref{Definition. good valuations} (4) holds. We conclude that $\fv$ is a good valuation. 
\end{proof}


\bibliographystyle{alpha}
\bibliography{ref}

\end{document}